\documentclass[12pt]{amsart}
\usepackage[utf8]{inputenc}
\usepackage{amsmath, amssymb}
\usepackage{amscd}
\usepackage{amsthm,amsmath}
\usepackage{amsfonts}
\usepackage{latexsym}
\usepackage{graphicx}
\usepackage[all,tips, 2cell]{xy}
\usepackage{bbm} 
\UseTwocells
\usepackage{verbatim}
\usepackage[margin=1.1in]{geometry}
\usepackage{hyperref}

\DeclareFontFamily{U}{mathc}{}
\DeclareFontShape{U}{mathc}{m}{it}%
{<->s*[1.03] mathc10}{}
\DeclareMathAlphabet{\mathcal}{U}{mathc}{m}{it}

\newtheorem{theorem}{Theorem}[section]
\newtheorem{lemma}[theorem]{Lemma}
\newtheorem{proposition}[theorem]{Proposition}
\newtheorem{corollary}[theorem]{Corollary} 
\theoremstyle{definition}  

\newtheorem{example}[theorem]{Example}

\newtheorem{remark}[theorem]{Remark}

\newcommand{\id}{\text{id}} 
\newcommand{\FPdim}{\text{FPdim}}

\renewcommand{\Vec}{\operatorname{\operatorname{\mathsf{Vec}}}}

\newcommand{\Pic}{{{Pic}}}

\newcommand{\Picbr}{Pic_{br}}

\newcommand{\Aut}{Aut}

\DeclareMathOperator{\Coker}{\operatorname{\mathsf{Coker}}}

\newcommand{\Inv}{Inv}

\DeclareMathOperator{\Ind}{Ind}
\DeclareMathOperator{\Hom}{\operatorname{\mathsf{Hom}}}
\DeclareMathOperator{\Quad}{\operatorname{\mathsf{Quad}}}
\DeclareMathOperator{\Ext}{\operatorname{\mathsf{Ext}}}

\DeclareMathOperator{\Ker}{\operatorname{\mathsf{Ker}}}

\DeclareMathOperator{\Sym}{\operatorname{\mathsf{Sym}}}

\newcommand{\Exctr}{Ex_{ctr}}
 
\newcommand{\uuExctr}{\mathbf{Ex_{ctr}}}

\newcommand{\Mext}{{Mext}}
\newcommand{\uMext}{\mathcal{M\mkern-3mu ext}} 
\newcommand{\uuMext}{\mathbf{Mext}}

\newcommand{\uPic}{\mathcal{P\mkern-3mu ic}} 
\newcommand{\uuPic}{\mathbf{Pic}}

\newcommand{\uPicbr}{\mathcal{{P\mkern-3mu ic}_{br}}} 
\newcommand{\uuPicbr}{\mathbf{{Pic}_{br}}}

\newcommand{\uAut}{\mathcal{A\mkern-3mu ut}} 

\newcommand{\uInv}{\mathcal{I\mkern-3mu nv}}

\newcommand{\alt}{\text{alt}}
\newcommand{\rev}{\text{rev}}

\newcommand{\kk}{\mathbbm{k}}

\newcommand{\eps}{\varepsilon}

\newcommand{\B}{\mathcal{B}}
\newcommand{\C}{\mathcal{C}}

\newcommand{\I}{\mathcal{I}}

\newcommand{\D}{\mathcal{D}}
\newcommand{\E}{\mathcal{E}}

\newcommand{\Z}{\mathcal{Z}}

\newcommand{\M}{\mathcal{M}}
\newcommand{\A}{\mathcal{A}}

\newcommand{\be}{\mathbf{1}}

\renewcommand{\be}{\mathbf{1}}

\newcommand{\bt}{\boxtimes}
\newcommand{\ot}{\otimes}

\newcommand{\Fun}{Fun}
\newcommand{\uFun}{\mathcal{F\mkern-3mu un}}
\newcommand{\TwoFun}{2\mhyphen Fun}
\newcommand{\uTwoFun}{\mathcal{2\mhyphen F\mkern-3mu un}}

\newcommand{\Vect}{\mathcal{V \mkern-3mu ect}} 
\newcommand{\sVect}{\mathcal{sV \mkern-3mu ect}} 
\newcommand{\Rep}{\mathcal{R\mkern-3mu ep}}

\newcommand\void[1]{}

\mathchardef\mhyphen="2D
\newcommand{\FUN}{\mathbf{2\mhyphen}\mathbf{Fun}}

\newcommand{\bG}{\mathbf{G}}

\begin{document}
%
\title[Minimal non-degenerate extensions  of a super-Tannakian  category]{Computing the group of minimal non-degenerate extensions  of a super-Tannakian  category}
\author{Dmitri Nikshych}
\address{Department of Mathematics and Statistics,
University of New Hampshire,  Durham, NH 03824, USA}
\email{dmitri.nikshych@unh.edu}
\date{}
\begin{abstract}
We prove an analog of the K\"unneth formula for the groups of minimal non-degenerate extensions \cite{LKW1}  of  
symmetric fusion categories. We describe in detail the structure of the group of minimal extensions of a pointed super-Tannakian
fusion category. This description resembles that of the third cohomology group of a finite abelian group. We explicitly compute
this group in several concrete examples.
\end{abstract}

\maketitle

%


\section{Introduction }

There is a notion of categorical ``orthogonality''  in a braided fusion category $\C$  \cite{DGNO1, Mu}. Namely,  
objects $X,Y$ of $\C$ {\em centralize}  each other if the squared braiding between them is identity, i.e., $c_{Y,X}c_{X,Y} = \id_{X\ot Y}$,
where $c$ denotes the braiding of $\C$.  For a fusion subcategory $\B\subset \C$,  its {\em centralizer}
 in $\B$ is the fusion subcategory $\B'\subset \C$ consisting of all objects $X$ centralizing every object $Y$  in $\B$.
 When $\C$ is pointed, i.e., corresponds to a pre-metric group $(A,\,q)$, where $q$ is a quadratic form on a finite Abelian group $A$, 
 fusion subcategories of $\B$ are in bijection with subgroups of $A$   and the centralizers are identified with  orthogonal complements.
 This construction allows interpreting  many aspects of the theory of braided fusion categories in terms of ``categorical linear algebra".  
 
 For example,
 a symmetric fusion subcategory $\E\subset \C$  satisfies $\E \subseteq \E'$ and so can be thought of as a categorical analog of a coisitropic subspace.
 When $\C$ is non-degenerate, one has  $\E=\E'$, i.e., $\E$ is a {\em Lagrangian} subcategory, if and only if 
  $\FPdim(\E)^2 = \FPdim(\C)$, where $\FPdim$ denotes the Frobenius-Perron dimension. An embedding $\E \hookrightarrow \C$
with this property  will be  called a {\em minimal non-degenerate extension} (or simply a {\em minimal  extension}) of $\E$.
 Lan, Kong, and Wen observed in  \cite{LKW1}  that  there is a natural product of minimal extensions of $\E$, so that the set of their
 equivalence classes is an abelian group $\Mext(\E)$ (in fact, minimal extensions of $\E$ form a symmetric $2$-categorical group 
 $\uuMext(\E)$).  

Groups of minimal extensions were studied by many authors, including \cite{ABK, BCHM, BGHNPRW,  DNR, GVR, GS, LKW1, LKW2, VR}.
From the physics point of view, minimal extensions appear in the description of  2+1D topological orders  and symmetry-protected trivial 
(SPT) orders \cite{LKW1}, and  symmetric invertible fermionic phases \cite{ABK, BCHM}.
 It is known that for  a Tannakian category $\E= \Rep(G)$, where $G$ is a finite group,  the group $\Mext(\E)$ is isomorphic to 
 $H^3(G,\, \kk^\times)$,  the third cohomology group of $G$. 
 For  $\E=\sVect$, the category of super-vector spaces,
 it is isomorphic to $\mathbb{Z}/16\mathbb{Z}$ (this statement is known in physics as Kitaev's $16$-fold way \cite{Kit}).
 
 The goal of this paper is to describe the group of minimal extensions of a super-Tannakian fusion category $\E$ (with an emphasis on the case
 when $\E$ is pointed) and compute it in several concrete examples.
 
 The main results of the present paper are the following.
 
 In Section~\ref{Section Kunneth} we prove a version of the K\"unneth formula for $\Mext(\Rep(G) \bt \E)$, where $G$ is a finite group
 and $\E$ is a symmetric fusion category.   Theorem~\ref{Kunneth general}  establishes a group isomorphism
\begin{equation}
\Mext(\Rep(G)\bt \E)  \cong \Mext(\E) \times \TwoFun(G,\, \uuPic(\E)),
\end{equation}
where  $\TwoFun(G,\, \uuPic(\E))$ is the group of monoidal 2-functors from $G$ to the $2$-categorical Picard group of $\E$. For $\E=\Rep(L)$
this recovers  the familiar K\"unneth formula computing the third cohomology of the product $G\times L$.

In Section~\ref{Section pointed} we analyze the structure of the group $\Mext(\E)$ for a pointed symmetric category $\E$.
We consider a filtration
\begin{equation}
\Mext_{triv}(\E) \subset \Mext_{pt}(\E) \subset  \Mext_{int}(\E) \subset \Mext(\E),
\end{equation}
 consisting, respectively,  of the subgroups of  trivial, pointed, and integral minimal extensions of $\E$, and compute its composition factors
 in Theorem~\ref{factors}. This description of $\Mext(\E)$  generalizes that of the third cohomology group $H^3(A,\,\kk^\times)$ 
 of a finite abelian group $A$ \cite{DS, MN}. A new feature is the appearance of cohomological obstructions from the theory of graded extensions \cite{DN2, ENO}.
 

Finally, in Section~\ref{examples section}  we apply our results to compute  the group of minimal extensions of 
concrete examples of super-Tannakian categories,  namely, $\Rep(\mathbb{Z}_{2^n}^f)$ and $\Rep(\mathbb{Z}_{2} \times \mathbb{Z}_2^f)$.

\subsection{Acknowledgements} 

The author is grateful  to David Aasen, Maissam Barkeshli, Alexei Davydov, Pavel Etingof, and Victor Ostrik for useful discussions.  
The author's work  was supported  by  the  National  Science  Foundation  under  
Grant No.\ DMS-1801198.  

%
 
\section{Preliminaries}

In this paper, we work over an algebraically closed field $\kk$ of characteristic $0$. We adapt the following font convention
for higher categorical groups:  we use italics ($G$) to denote ordinary groups, calligraphic ($\mathcal{G}$) for categorical groups,
and boldface $\mathbf{G}$ for $2$-categorical groups. We use the same name in different fonts to denote the truncations
of a given $2$-categorical group. For example, we write $\Pic(\E)=\pi_0(\uuPic(\E))$ and  $\uPic(\E)=\pi_{\leq 1}(\uuPic(\E))$
for the truncations of the $2$-categorical Picard group $\uuPic(\E)$.

\subsection{Symmetric fusion categories and their Picard groups}
\label{sect Picard groups}

We refer the reader to \cite{DGNO2,EGNO} for the basics of the theory of braided fusion categories.

By Deligne's theorem \cite{D},  symmetric fusion categories are parameterized by pairs $(G,\,t)$,
where $G$ is a finite group and $t\in Z(G)$ is a central element such that $t^2=1$. 
The corresponding category $\Rep(G,\,t)$ consists of finite-dimensional
representations of  $G$, with the usual tensor product and braiding given by
\begin{equation}
c_{V,W}(v\ot w) =
\begin{cases}
-w \ot v, & \text{if $t|_V=-1$ and  $t|_W =-1$} ,\\
w \ot v, & \text{otherwise,}
\end{cases} 
\end{equation}
for all irreducible representations $V,\, W$ of $G$, where $v\in V,\, w\in W$.
This category is called {\em Tannakian} if $t=1$ (so it is simply $\Rep(G)$) and {\em super-Tannakian}
if $t\neq 1$. When $G$ contains a unique automorphism orbit of central elements of order $2$
we will use notation $\Rep(G^f)$ for $\Rep(G,\,t)$ which is common in physics. Here $f$ stands for ``fermionic".
For example, $\sVect = \Rep(\mathbb{Z}_2^f)$.

Let $\B$ be  a braided fusion category. The $2$-categorical {\em Picard group} $\uuPic(\B)$ \cite{ENO} is formed by
invertible $\B$-module categories with the tensor product $\bt_\B$. Its $1$-categorical truncation $\uPic(\B):=\pi_{\leq 1}(\uuPic(\B))$ is equivalent  to the
categorical group $\uAut(\Z(\B);\B)$ of braided tensor autoequivalences of $\Z(\B)$ trivializable on $\B$ \cite{DN1}.

The Picard group of a symmetric fusion category was determined by Carnovale  in \cite{C}:
\begin{equation}
\label{Carnovale}
\Pic(\Rep(G,\,t)) =
\begin{cases}
H^2(G,\,\kk^\times) & \text{if $t=1$},\\
H^2(G,\,t,\,\kk^\times) & \text{if $t\neq 1$ and  $\langle t \rangle$ is not a direct summand of $G$},\\
H^2(G,\,t,\,\kk^\times) \times \mathbb{Z}_2 & \text{if $t\neq 1$ and  $\langle t \rangle$ is  a direct summand of $G$.}
\end{cases}
\end{equation}
Here the elements of $H^2(G,\,\kk^\times)$ correspond to module categories $\Rep(k_\mu[G])$ of projective representations
of $G$ with a fixed $2$-cocycle $\mu \in Z^2(G,\,\kk^\times)$ and the generator of $\mathbb{Z}_2$ is $\Rep(G_0)$,
where $G =G_0 \times \langle t \rangle$.

The  group $H^2(G,\,t,\,\kk^\times)$ is defined as follows (see  \cite{C}  and also \cite{DN2}).
There is a canonical bilinear map
\begin{equation}
\label{form xi}
H^2(G,\, \kk^\times) \times G\to \mathbb{Z}/ 2\mathbb{Z} : (\mu,\, x) \mapsto \xi_\mu(x),
\end{equation}
where $\xi_\mu(x)$ is defined by the condition 
\[
(-1)^{\xi_\mu(x)} = \frac{\mu(x,\,t)}{\mu(t,\,x)},\qquad \mu \in H^2(G,\, \kk^\times),\, x\in G.
\]
Introduce a new multiplication on $H^2(G,\,\kk^\times)$ by 
\begin{equation}
\label{product *}
\mu * \nu (x,\,y)  = \mu(x,\,y) \nu(x,\,y)\, (-1)^{\xi_\mu(x) \xi_\nu(y)}, \qquad x,y\in G,
\end{equation}
on representatives $\mu,\, \nu$ of cohomology classes in $H^2(G,\,\kk^\times)$.  The resulting group
will be denoted $H^2(G,\,t,\, \kk^\times)$. It is non-canonically isomorphic to $H^2(G,\,\kk^\times)$,
see \cite{C} for details. 

There is a subgroup  $\Pic(\Rep(G,\,t))_{int} \subset \Pic(\Rep(G,\,t))$ consisting of {\em integral} module categories, i.e.,
those in which all objects have integral Frobenius-Perron dimension. We have  
\begin{equation}
\label{Picint = H2 Gtk}
\Pic(\Rep(G,\,t))_{int} = H^2(G,\,t,\,\kk^\times).
\end{equation}

The {\em braided $2$-categorical Picard group}  $\uuPicbr(\E)$  \cite{DN2} of 
 a symmetric fusion category $\E$  consists of invertible objects in the $2$-center of a monoidal $2$-category
 of $\E$-module categories. The underlying braided categorical group   $\uPicbr(\E)$ was described in \cite[Section 6]{DN2}.  
 One has $\Picbr(\E) \cong \Pic(\E)\times  \Aut_\ot(\id_\E)$  and the corresponding quadratic form
 \begin{equation}
 \label{Quadratic from Q}
 Q_\E: \Picbr(\E)  \to \Inv(\E)
 \end{equation}
 is explicitly computed in \cite[Proposition 6.11]{DN2}.  In particular,  for $\E=\Rep(G,\,t)$ the integral part
 of $\Picbr(\E)$ is $H^2(G,\,t,\, \kk^\times) \times Z(G)$,
 where $Z(G)$ denotes the center of $G$.  In this case, $\Aut_\ot(\id_\E) = \Hom(G,\, \kk^\times)$ and 
 the restriction of the quadratic form \eqref{Quadratic from Q} on the integral part of $\Picbr(\E)$ is given  by 
 \begin{equation}
 \label{Qint}
Q_\E (\mu,\, z) = \frac{ \mu(zt^{\xi_\mu(z)+1},\, -)} {\mu( -,\, zt^{\xi_\mu(z)+1})}
\end{equation}
for all $\mu \in H^2(G,\,t,\, \kk^\times) $ and $z\in Z(G)$.

\subsection{The $2$-categorical group of minimal extensions of a symmetric  category}

We recall the definition given in \cite{LKW1}.  Let $\E$ be a symmetric fusion category. An embedding $\E\hookrightarrow \C$
into a non-degenerate  braided fusion category $\C$ is called a {\em minimal non-degenerate extension} (or, simply, a 
{\em minimal extension}) of $\E$ if the latter coincides with its centralizer in $\C$, i.e., $\E=\E'$, where
\[
\E' = \{X \in \C \mid c_{YX}\circ c_{XY} =\id_{X\ot Y} \mbox{ for all } Y\in \E\}.
\]
This condition is equivalent to the equality $\FPdim(\C)=\FPdim(\E)^2$, where $\FPdim$ denotes the Frobenius-Perron dimension of
the category.  Minimal extensions of $\E$ form a $2$-groupoid $\uuMext(\E)$.
An isomorphism between minimal extensions $\E\hookrightarrow \C_1$ and $\E\hookrightarrow \C_2$ 
is a braided equivalence $\C_1 \xrightarrow{\sim} \C_2$ that restricts to the identity on $\E$. A $2$-isomorphism is 
a natural isomorphism of equivalences, again identical on $\E$.

There is a natural tensor product of minimal extensions of $\E$. Namely, let $\E \hookrightarrow \C_1$ and $\E \hookrightarrow \C_2$
be minimal extensions. Then   $\E\bt\E$ embeds into $\C_1 \bt \C_2$. Since $\E$ is symmetric,
the tensor product  $\ot: \E \bt \E \to \E$ is a braided tensor functor. Its adjoint sends the unit object $\be$ to an \'etale (i.e., separable commutative)
 algebra $A \in \E \bt \E$.
The fusion category $(\C_1\bt\C_2)_A$  of $A$-modules in $\C_1\bt\C_2$ 
contains a {\em braided} fusion subcategory 
\[
\C_1\boxdot \C_2 :=(\C_1\bt\C_2)_A^0
\]
of {\em local} modules,
i.e., of  $A$-modules $(V,\, \rho: A\ot V \to V)$ with $c_{X,A}c_{A,X}\rho = \rho$. Note that 
$\E \cong (\E \bt \E)_A$  is embedded into  $\C_1\boxdot \C_2$.
The resulting embedding $\E \hookrightarrow \C_1\boxdot \C_2$ is, by definition,  the tensor product of $\E \hookrightarrow \C_1$  and $\E \hookrightarrow \C_2$. 
The unit object for this tensor product  is $\E \hookrightarrow \Z(\E)$, the embedding of $\E$ into its Drinfeld center.
The inverse of the embedding $\E\hookrightarrow \C$ is $\E\hookrightarrow \C^\rev$, where $\C^\rev$ coincides with $\C$ as a fusion
category with a braiding obtained by reversing the braiding of $\C$, namely $c^\rev_{X,Y}=c_{Y,X}^{-1},\,X,Y\in \C$.  
We refer the reader to \cite[Section 4.2]{LKW1} for details.

As agreed above, we denote  $\Mext(\E) =\pi_0(\uuMext(\E))$ and $\uMext(\E) =\pi_{\leq 1}(\uuMext(\E))$, the truncations
of $\uuMext(\E)$.

\begin{example}
\label{RepG example}
Let $G$ be a finite group. The group $\Mext(\Rep(G))$ was computed in \cite[Section 4.3]{LKW1}. Namely, a typical  element of  this group
is a  twisted  Drinfeld double of $G$:
\[
\Rep(G) \hookrightarrow \Z(\Vec_G^\omega),\qquad \omega \in Z^3(G,\,\kk^\times). 
\]
The product of these extensions corresponds to the product of $3$-cocycles.  Furthermore,
extensions corresponding to $3$-cocycles $\omega_1,\, \omega_2$ are isomorphic if and only if $\omega_1,\, \omega_2$ are cohomologous.
 Thus,
\begin{equation}
\label{Mext RepG}
\Mext(\Rep(G)) \cong H^3(G,\,\kk^\times). 
\end{equation}
This result can also be deduced from \cite[Section 4.4.10]{DGNO1} since for any minimal extension $\Rep(G) \hookrightarrow \C$, the 
image of $\Rep(G)$ is a Lagrangian subcategory of $\C$.
\end{example}

\begin{example}
\label{sVect example}
It was shown independently in \cite{BGHNPRW},  \cite[Proposition 5.14]{DNO},   \cite{Kit}, and \cite[Theorem 4.25]{LKW1}  that 
\begin{equation}
\label{Mext sVect}
\Mext(\sVect) \cong \mathbb{Z}_{16}. 
\end{equation}
This statement is known as Kitaev's $16$-fold way \cite{Kit}.  Any {\em Ising} category, i.e., a non-pointed braided fusion category of
dimension $4$ \cite[Appendix B]{DGNO2}, is a generator of this group.  Other elements of $\Mext(\sVect)$ are pointed braided fusion
categories coming from metric groups $(A,\,q)$ of order $4$  such that there exists $u\in A$ with $q(u)=-1$.

The isomorphism  \eqref{Mext sVect} can be identified with
\begin{equation}
\label{central charge 16}
\Mext(\sVect) \xrightarrow{\sim} \{\xi\in \kk^\times \mid \xi^{16}=1\}: \C \mapsto \xi(\C),
\end{equation}
where $\xi(\C)$ is the central charge of the category $\C$. Thus, the class of a minimal non-degenerate extension of $\sVect$ is 
completely determined by its central charge.
\end{example}

Let $\Rep(G,\,t)$ be a super-Tannakian category.  It contains a unique maximal Tannakian subcategory $\mathcal{T}= \Rep(G/\langle t \rangle)$.
If $\Rep(G,\,t) \hookrightarrow \C$ is a minimal extension then the de-equivariantization $\C^0:=\mathcal{T}' \bt_{\mathcal{T}} \Vect$ is a minimal
extension of $\sVect$.
The assignment  
\begin{equation}
\label{wGt}
w_{(G,t)}: \Mext(\Rep(G,\,t))  \to \Mext(\sVect)  :\C \mapsto  \C^0
\end{equation} 
is a group homomorphism \cite[Section 5.2]{LKW1}. As in \eqref{central charge 16},
this homomorphism is identified with taking the central charge.

By \cite[Corollary 4.9]{GVR}, $w_{(G,t)}$ is surjective if and only if $\langle t \rangle$ is a direct summand of $G$. In  this case,
this homomorphism splits, i.e., $ \Mext(\sVect)$ is a direct summand of $\Mext(\Rep(G,\,t))$.

\subsection{Central and braided   graded extensions}
\label{prelim ext}

Let $\B$ be a braided fusion category. Let $G$ be a finite group. A {\em central} $G$-graded extension of $\B$ is a $G$-graded
fusion category 
\begin{equation}
\label{displayed extension}
\C =\bigoplus_g\, \C_g,\qquad \C_e=\B,
\end{equation}
along with an embedding $\B \hookrightarrow \Z(\C)$.  It was shown in \cite{JMPP}
that  a  central $G$-graded extension   is the same thing as a $G$-crossed braided extension. 
By \cite{DN2, ENO} the $2$-groupoid of such extensions is equivalent to the $2$-groupoid of monoidal $2$-functors
$G \to \uuPic(\B)$. For such an extension there is canonical action 
of $G$ on the trivial component given by the composition
\begin{equation}
\label{action on the trivial component}
G \to \uPic(\B) \to \uAut^{br}(\B),
\end{equation}
where the first functor corresponds to the graded extension \eqref{displayed extension} and the second one is the canonical monoidal
functor associated to $\B$ \cite{DN1}.

Now let $A$ be a finite Abelian group. Braided $A$-graded extensions of a braided fusion category $\B$ were  classified
in \cite{DN2}.  The $2$-groupoid of such extensions is equivalent to the $2$-groupoid of braided monoidal $2$-functors
$A \to \uuPicbr(\B)$, where the latter is the braided $2$-categorical group of invertible {\em braided} $\B$-module categories.

\section{Central graded extensions and the K\"unneth formula}
\label{Section Kunneth}

\subsection{The group of monoidal $2$-functors to a braided $2$-categorical group}

Let $G$ be a  group and let $\mathcal{G}$ be a  braided categorical group.
Let $C,\,C': G \to \mathcal{G}$ be 
monoidal functors, where $C$ is given by  $x\mapsto \C_x$  with the monoidal
structure $M_{x,y}: \C_x\ot \C_y \xrightarrow{\sim} \C_{xy}$  and $C'$ is given by  $x\mapsto \C'_x$  with the monoidal
structure $M'_{x,y}: \C'_x\ot \C'_y \xrightarrow{\sim} \C'_{xy}$, $x,y\in G$.   Clearly, such functors must factor through the
commutator subgroup of $G$.

Define  a monoidal functor 
\begin{equation}
\label{tildeM}
\tilde{C}:=
C\ot C': G \to  \mathcal{G}: \qquad x\mapsto \C_x \ot \C'_x.
\end{equation}
with the monoidal  structure 
\begin{equation}
\tilde{M}_{x,y}:  \C_x \ot \C'_x \ot  \C_y \ot \C'_y \xrightarrow{B_{x',y}} \C_x \ot \C_y \ot  \C'_x \ot  \C'_y 
\xrightarrow{M_{x,y} \ot  M'_{x,y}}  \C_{xy} \ot   \C'_{xy},\qquad x,\,y\in G.
\end{equation}
Here $B_{x',y}$ denotes the braiding in $\mathcal{G}$ between $\C'_x$ and $\C_y$.  With this product, the isomorphism
classes of such braided monoidal functors form a categorical group which we denote $\uFun(G,\,\mathcal{G})$. 
The identity element of this group is the trivial functor and the inverse of $C: x\mapsto \C_x$ is  $C^{-1}: x\mapsto \C^{-1}_x$.
If $\mathcal{G}$  is symmetric, the underlying group $\Fun(G,\,\mathcal{G})$ is Abelian.

There is an obvious short exact sequence
\begin{equation}
\label{sequence 1}
0 \to H^2(G,\, \pi_1(\mathcal{G})) \to \Fun(G,\,\mathcal{G}) \to \Hom(G,\,\pi_0(\mathcal{G})) \to 0.
\end{equation}
Here $H^2(G,\, \pi_1(\mathcal{G}))$ is isomorphic to the group  of monoidal functor structures on the trivial functor.
This sequence does not split in general because the braiding of $\mathcal{G}$ may be non-trivial.

Now let $\bG$  be a braided $2$-categorical group.  Let $\pi_0(\bG),\,\pi_1(\bG),\,\pi_2(\bG)$ denote, respectively, the group
of invertible objects of $\bG$, the group of automorphisms of the unit object $\be_\bG$, and the group of automorphisms of $\id_{\be_\bG}$.
Let $\pi_{\leq 1}(\bG)$ denote the braided categorical group obtained by truncating $\bG$, whose objects are objects of $\bG$ and morphisms
are isomorphism classes of $1$-cells in $\bG$.  Let $\pi_{\geq 1}(\bG)$ denote the symmetric categorical group of $1$-automorphisms
of the unit object of $\bG$.

There is a $2$-categorical analog of the above construction of a categorical group of monoidal functors.
It was explained in \cite[Section 2.8]{DN2}  that isomorphisms classes of monoidal 
$2$-functors from $G$ to $\bG$ also  form a categorical group denoted $\uTwoFun(G,\, \bG)$.   Namely, if $C,\,C': G \to \bG$ are such 
$2$-functors, then the monoidal structure of  the product $C\ot C'$ is defined as above and the structural associativity $2$-cells are given by
\cite[diagram (2.78)]{DN2}.  These cells involve  the  associativity $2$-cells of $C$ and $C'$ and the structure $2$-cells of $\bG$.

The group $\TwoFun(G,\, \bG)$ fits into the following exact sequence \cite[Theorem 2.38]{DN2}:
\begin{equation}
\label{exact sequence Funsym}
H^1(G,\,\pi_1(\bG)) \xrightarrow{\alpha}  H^3(G,\, \pi_2(\bG)) \xrightarrow{\beta} \TwoFun(G,\, \bG)
  \xrightarrow{\gamma}  \Fun(G,\, \pi_{\leq 1}(\bG)) \xrightarrow{\delta} \ H^4(G,\,\pi_2(\bG)).
\end{equation}
Here $\alpha$ assigns to a homomorphism $G \to \pi_1(\bG)$
the corresponding pullback of the associator of the categorical group $\pi_{\geq 1}(\bG)$ (the latter is an element of  $H^3(\pi_1(\bG),\, \pi_2(\bG))$),
$\beta$ assigns to a third cohomology class the monoidal $2$-functor structure on the trivial $2$-functor, $\gamma$ assigns
to a monoidal $2$-functor $G \to \bG $ the underlying $1$-functor to the truncation of $\bG$, and $\delta$ gives the obstruction
for a given monoidal functor to extend to a monoidal  $2$-functor.

\begin{example}
\label{G=Pic(E)}
Let $\E$ be a symmetric fusion category and let $\bG=\uuPic(\E)$ be the symmetric $2$-categorical Picard group of $\E$. 
We have  $\pi_{\leq 1}(\uuPic(\E)) = \uPic(\E)$, the symmetric categorical Picard group of $\E$, and 
$\pi_{\geq 1}(\uuPic(\E)) = \uInv(\E)$,  the symmetric categorical group of invertible objects of $\E$.
Since the associator of the latter is trivial, the exact sequence \eqref{exact sequence Funsym} becomes
\begin{equation}
\label{exact sequence Funsym for Picard}
0 \xrightarrow{\alpha}  H^3(G,\, \kk^\times) \xrightarrow{\beta} \TwoFun(G,\, \uuPic(\E))
  \xrightarrow{\gamma}  \Fun(G,\, \uPic(\E)) \xrightarrow{\delta} \ H^4(G,\,\kk^\times).
  \end{equation}
\end{example}

\begin{example}
\label{obstruction for Pic(E)}
The restriction of the obstruction map $\delta$ from Example~\ref{G=Pic(E)} to the  subgroup $H^2(G,\, \Inv(\E))$ of $\Fun(G,\, \uPic(\E))$  
is given by the composition
\begin{equation}
\label{cup composition}
H^2(G,\, \Inv(\E)) \xrightarrow{q} H^2(G,\, \mathbb{Z}/ 2\mathbb{Z})  \xrightarrow{\cup^2}  H^4(G,\, \mathbb{Z}/ 2\mathbb{Z})  \xrightarrow{\iota} H^4(G,\, \kk^\times),
\end{equation}
where $q:  \Inv(\E) \to \mathbb{Z}/ 2\mathbb{Z}$ is the quadratic homomorphism of the symmetric categorical group $\Inv(\E)$, 
$\cup^2:H^2(G,\, \mathbb{Z}/ 2\mathbb{Z})  \to  H^4(G,\, \mathbb{Z}/ 2\mathbb{Z})$   is the cup square in $H^*(G,\, \mathbb{Z}/ 2\mathbb{Z})$
(note that it is a homomorphism), 
and $\iota$ is induced by 
the inclusion of the coefficients $\mathbb{Z}/ 2\mathbb{Z}  \to \kk^\times: n\mapsto (-1)^n$.
Indeed, by \cite[Section 8.7]{ENO} (see also \cite[(8.45)]{DN2}),  the obstruction $\delta(L)$ for $L \in H^2(G,\, \Inv(\E))$ is given by 
\begin{equation}
\label{delta L}
\delta(L)(x,y,z,w) = c_{L(x,y),L(z,w)},\qquad x,y,z,w \in G,
\end{equation}
where $c$ is the braiding in $\Inv(\E)$.  Since $c_{X,Y}=(-1)^{q(X)q(Y)}$ for all $X,Y \in \Inv(\E)$, the fiormula \eqref{delta L} translates to 
\eqref{cup composition}.

For $\E=\sVect$ formula \eqref{cup composition} describing the obstruction map 
\[
H^2(G,\, \Inv(\E)) = H^2(G,\, \mathbb{Z}/ 2\mathbb{Z})   \to H^4(G,\, \kk^\times)
\] 
in terms of the cup product  appeared in \cite[Section VI.B]{ABK}.
\end{example}

\subsection{The embedding $\uuPic(\E) \hookrightarrow \uuPic(\Z(\E))$}
\label{section PIcE into PicZE}

Let $\E$ be a symmetric fusion category. The induction $2$-functor
\begin{equation}
\label{induction from Pic}
\Ind: \uuPic(\E) \to \uuPic(\Z(\E)) :\M \mapsto \Z(\E) \bt_\E \M
\end{equation}
is a monoidal $2$-embedding of categorical groups. On the level of $1$-cells this functor embeds $\Inv(\E)$ into $\Inv(\Z(\E))$.

There is an equivalence of categorical groups \cite{DN2,ENO}
\begin{equation}
\label{partial}
\partial: \pi_{\leq 1} (\uuPic(\Z(\E)) \to  \uAut^{br}(\Z(\E)) : \M \mapsto \partial(\M), 
\end{equation}
such that $\alpha_\M^+ \cong \alpha_\M ^- \circ \partial(\M)$, where $\alpha_\M ^\pm : \Z(\E) \xrightarrow{\sim} \uFun_{\Z(\E)}(\M,\,\M)$
are two equivalences defined, respectively, using the braiding of $\Z(\E)$ and its reverse.  This equivalence $\partial$ sends  $Z\in \Inv(\Z(\E))$
to $\delta(Z) \in \id_{\Z(\E)}$ defined by  $\partial(Z)_X= c_{Z,X}c_{X, Z}$ for all $X\in \Z(\E)$.

It was shown in \cite{DN2} that under the equivalence \eqref{partial},  in the  image of $\Ind$ the objects  are autoequivalences $\alpha \in  \Aut^{br}(\Z(\E)) $
such that  $\alpha|_\E\cong \id_\E$ and automorphisms of the unit object   are $\nu \in \id_{\Z(\E)}$ such that $\nu_X=1$ for all $X\in \E$.

For a monoidal $2$-functor 
\begin{equation}
\label{action M}
M: G \to \uuPic(\Z(\E)): g \mapsto \M(g)
\end{equation}
we denote $\partial(g)=\partial(\M(g))\in \Aut^{br}(\Z(\E))$ and $\partial_{g,h}: \partial(g) \circ \partial(h)
\to \partial(gh),\, g,h\in G,$ the monoidal structure on the composition functor $G \to  \Aut^{br}(\Z(\E))$. Note that $\partial_{g,h}$ is canonically identified
with an element of $\Aut(\id_{\Z(\E)}) =  \Aut(\partial(gh))$.

\begin{lemma}
\label{factor through PicE}
Suppose that the monoidal $2$-functot  \eqref{action M} is such that $\partial(g)|_\E \cong \id_\E$ and $(\partial_{g,h})|_X =\id_X$ for all $X \in \E$.
Then  \eqref{action M}  factors through the embedding \eqref{induction from Pic}:
\[
G \to \uuPic(\E) \xrightarrow{\Ind} \uuPic(\Z(\E)).
\]
\end{lemma}
\begin{proof}
This follows from the above description of the image of $\Ind$.
\end{proof}

\subsection{The group of central extensions of a symmetric fusion category}

Recall that  for  an \'etale algebra $A$ in $\Z(\C)$, where $\C$ is a fusion category,  the category $\C_A$ of $A$-modules in $\C$ is a fusion category.
If $A$ is an \'etale algebra in $\Z_{sym}(\B)$, where $\B$ is a braided fusion category,  then  $\B_A$ is a 
braided fusion category. The tensor product over a symmetric fusion category $\E$ is a special case of this construction. Indeed, the tensor product
$\ot: \E \bt \E \to \E$ is a braided tensor functor. Let $I: \E \to  \E \bt \E $ be its adjoint, then $A:=I(1)$ is a canonical \'etale algebra in $\E\bt \E$.
If $\E \hookrightarrow \C_1,\, \E \hookrightarrow \C_2$ are central inclusions of $\E$ into fusion categories $\C_1,\, \C_2$
then $\C_1 \bt_\E \C_2 = (\C_1 \bt \C_2)_A$ is a fusion category.
If $\E \hookrightarrow \B_1,\, \E \hookrightarrow \B_2$ are inclusions into symmetric centers of braided fusion categories $\B_1,\, \B_2$
then $\B_1 \bt_\E \B_2 = (\B_1 \bt \B_2)_A$ is a braided fusion category.

\begin{lemma}
\label{A-modules}
Let $\B \subset \C$ be a central extension of a braided fusion category $\B$ and let $A$ be an etale algebra in $\Z_{sym}(\B)$.
Then $\B_A \subset \C_A$ is a central extension.
\end{lemma}
\begin{proof}
This is straightforward since the half-braiding between objects of $\B$ and $\C$ induces the one between objects of $\B_A$ and $\C_A$.
\end{proof}


Let $\uuPic(\E)$ denote the braided $2$-categorical Picard group  of $\E$ and let
$\FUN(G,\, \uuPic(\E))$  denote the $2$-categorical group  of monoidal functors from $G$ to $\uuPic(\E)$.

It was observed in \cite{DN2} that the $2$-groupoid  $\uuExctr(G,\, \E)$ of  central $G$-graded extensions of $\E$ is
a braided $2$-categorical group. The tensor product of $G$-graded central  extensions $\E\hookrightarrow\C^1$ and 
$\E\hookrightarrow\C^2$ is defined as follows. Let $A$ be the canonical  \'etale algebra in $\E\bt \E$ defined above. 
 Lemma~\ref{A-modules} applies to the central  $G$-graded extension $\E \bt \E
\hookrightarrow  \bigoplus_{g \in G} \, \C^1_g \bt \C^2_g$, so we obtain  a central $G$-graded  extension
\begin{equation}
\label{product of G-crossed}
\E \cong (\E\bt \E)_A \hookrightarrow  \bigoplus_{g \in G} \, (\C^1_g \bt \C^2_g)_A = \bigoplus_{g \in G} \, \C^1_g \bt _\E\C^2_g 
\end{equation}
which is the product of  extensions $\E\hookrightarrow\C^1$ and  $\E\hookrightarrow\C^2$.
We will denote this product by $\E\hookrightarrow \C_1\boxdot \C_2$.


\begin{theorem}
\label{2Fun =Ex thm}
There is a monoidal $2$-equivalence of $2$-categorical groups
\begin{equation}
\label{2Fun =Ex}
\FUN(G,\, \uuPic(\E)) \cong \uuExctr(G,\, \E).
\end{equation}
\end{theorem}
\begin{proof}
The  $2$-equivalence is established for {\em braided} $\E$ in \cite[Theorem 7.12]{ENO} (see also \cite[Theorem 8.13]{DN2}). 
Namely, a monoidal $2$-functor $G \to \Pic(\E) : g \mapsto \C_g$ gives rise to a $G$-graded central extension $\bigoplus_{g\in G}\, \C_g$.
For symmetric $\E$,  the monoidal structure of this $2$-equivalence  is evident since the tensor products in $\FUN(G,\, \uuPic(\E))$
and $ \uuExctr(G,\, \E)$ are defined by the very same formulas  (cf.\ \eqref{tildeM} and \eqref{product of G-crossed}) and have the
same associativity $2$-cells.
%
\end{proof}

\subsection{The center of a central extension of a symmetric fusion category}
\label{the center of C/E section}

Let $\E$ be a symmetric fusion category and let
\begin{equation}
\label{displayed extension of E}
\C =\bigoplus_{g\in G}\, \C_g,\qquad \C_e=\E,
\end{equation} 
be its central $G$-graded extension. It was shown in \cite{GNN} that $\Z(\C)$ is equivalent to
a $G$-equivariantization of the relative center $\Z_\E(\C)$.
The latter is  equivalent to the fusion category $\Z(\E) \bt_\E \C$. It is a central $G$-graded extension of $\Z(\E)$  corresponding to
the following monoidal $2$-functor 
\begin{equation}
\label{factor through Pic E}
G \to \uPic(\E) \xrightarrow{Ind}  \uPic(\Z(\E)),
\end{equation}
where the first functor corresponds to the central extension \eqref{displayed extension of E} and the second is the induction  \eqref{induction from Pic}.
The action of $G$ on $\Z(\E)$ is obtained by composing \eqref{factor through Pic E} with the canonical monoidal equivalence
$\uPic(\Z(\E)) \cong \uAut^{br}(\Z(\E))$. It follows from \cite{DN2} (see Section~\ref{sect Picard groups}) that this action restricts to the trivial action on 
$\E \subset \Z(\E)$. Therefore, $\Z(\C) = (\Z(\E) \bt_\E \C)^G$ is a minimal extension of the symmetric fusion category $\E^G = \Rep(G)\bt \E$.

\begin{proposition}
\label{Ext to Mext prop}
The assignment
\begin{equation}
\label{Ext to Mext}
\uuExctr(G,\, \E) \to \uuMext(\Rep(G)\bt \E) : \C \mapsto \Z(\C)
\end{equation}
is a monoidal $2$-functor.
\end{proposition}
\begin{proof}
For a braided fusion category $\B$ containing a subcategory equivalent to $\Rep(G)$,
 let $\B_G =\B \bt_{\Rep(G)} \Vect$ denote the corresponding de-equivariantization.
Using definitions of the tensor products in $\uuExctr(G,\, \E)$ and $ \uuMext(\Rep(G)\bt \E)$ we obtain equivalences
\begin{eqnarray*}
\Z(\C^1 \boxdot \C^2)_G 
&\cong& \bigoplus_{g\in G}\, \Z(\E) \bt_\E \C^1_g  \bt_\E \C^2_g \\
&\cong& \bigoplus_{g\in G}\, (\Z(\E)\bt \Z(\E))_A^0 \bt_\E \C^1_g  \bt_\E \C^2_g \\ 
&\cong& \bigoplus_{g\in G}\, \left( (\Z(\E) \bt_\E \C^1_g)  \bt  (\Z(\E) \bt_\E  \C^2_g)  \right)_A^0 \\
&\cong& \left(  \Z(\C^1) \boxdot \Z(\C^2) \right)_G 
\end{eqnarray*}
for all $G$-graded central extensions $\C^1,\, \C^2$ of $\E$. Here $\B_A^0$ denotes the category of local $A$-modules in $\B$.
The symbol $\boxdot$ stands for the tensor product  in  both $\uuExctr$ and $\uuMext$.
Taking equivariantizations we get a canonical equivalence 
\[
\Z(\C^1 \boxdot \C^2) \cong  \Z(\C^1) \boxdot \Z(\C^2)
\]
in  $\uuMext(\Rep(G)\bt \E)$ that equips the $2$-functor  \eqref{Ext to Mext} with a canonical monoidal  structure.
\end{proof}

\subsection{The K\"unneth formula.}
\label{subs Kunneth}

Let $G,\,L$ be finite groups. 

\begin{proposition} 
There is a split short exact sequence  
\begin{equation}
\label{Kunneth formula}
\begin{split}
0\to H^3(G,\, \kk^\times)\,\oplus \,H^3(L,\, \kk^\times)\,\oplus \,\left( \Hom(G,\, \kk^\times) \otimes \Hom(L,\, \kk^\times) \right)\to 
H^3(G\times L,\, \kk^\times)\to \\
\Hom(L,\, H^2(G,\, \kk^\times))\,  \oplus \, \Hom(G,\, H^2(L,\, \kk^\times))  \to 0.
\end{split}
\end{equation}
\end{proposition}
\begin{proof} 
Using the short exact sequence  $0\to \mathbb{Z}\to \mathbb{Q} \to \mathbb{Q}/\mathbb{Z} \to 0$, we obtain isomorphisms
\[
H^i(G,\, \kk^\times) \cong H^i(G,\, \mathbb{Q}/\mathbb{Z})\cong H^{i+1}(G,\,\mathbb{Z})
\] 
for any finite group $G$ and $i\ge 1$. Also, $H^1(G,\,\mathbb{Z})=0$. 
Therefore, the sequence in question is obtained from the usual K\"unneth formula for integral cohomology. 
\end{proof} 

Below we generalize the K\"unneth formula \eqref{Kunneth formula}. 
Namely, for a finite group $G$ and a symmetric fusion category $\E$, we explain how to compute the group  $\Mext(\Rep(G)\bt \E)$. 

\begin{theorem}
\label{Kunneth general} 
There is a group isomorphism
\begin{equation}
\Mext(\Rep(G)\bt \E)  \cong \Mext(\E) \times \TwoFun(G,\, \uuPic(\E))
\end{equation}
\end{theorem}
\begin{proof}
Let $\C$ be a minimal extension of $\Rep(G)\bt \E$. Let $\widetilde{\C}$  denote the centralizer of $\Rep(G)$ in $\C$.
The de-equivariantization $\widetilde{\C} \bt_{\Rep(G)} \Vect$  is a minimal  extension of $\E$ and the assignment
\begin{equation}
\label{Mext to Mext}
\uuMext(\Rep(G)\bt \E)  \to \uuMext(\E)  : \C \mapsto \widetilde{\C} \bt_{\Rep(G)} \Vect
\end{equation}
is a monoidal $2$-functor between categorical groups.

The associated group homomorphism 
\begin{equation}
\label{Mext to Mext hom}
\Mext(\Rep(G)\bt \E)  \to \Mext(\E)
\end{equation}
is  split surjective, since for any $\D\in \Mext(\E)$ we have a minimal extension
$\Rep(G)\bt \E  \hookrightarrow \Z(\Rep(G)) \bt \D$. 

Let $K(G,\, \E)$ denote the kernel of \eqref{Mext to Mext hom}. It remains to show that  $\TwoFun(G,\, \uuPic(\E)) = K(G,\, \E)$.
Theorem~\ref{2Fun =Ex thm} combined with Proposition~\ref{Ext to Mext prop} gives an inclusion 
\[
 \TwoFun(G,\, \uuPic(\E))  \cong \Exctr(G,\, \E)\to K(G,\, \E).
\]
Let us show that it is surjective.
%
A minimal extension $\Rep(G)\bt  \E \hookrightarrow \C$  is in $K(G,\,\E)$
 if and only if its de-equivariantization $\C_G =\C \bt_{\Rep(G)} \Vect$ is a $G$-graded central extension
of $\Z(\E)$ such that the action of $G$  restricts trivially to the subcategory $\E \subset \Z(\E)$.
By \cite{DN1}, this means that the corresponding monoidal $2$-functor $G\to \uuPic(\Z(\E))$ 
is the composition of a monoidal $2$-functor $F: G\to \uuPic(\E)$  and  the induction $2$-functor \eqref{induction from Pic}.
As explained in Section~\ref{the center of C/E section}, this means that $\C \cong \Z(\A_F)$, where $\A_F$
is the central extension of $\E$ corresponding to $F$. Thus, $\C \in K(G,\,\E)$, as required.
\end{proof}

\begin{remark}
\label{DGNO characterization remark}
Recall that a $G$-{\em gauging} of a braided fusion category $\B$ is the equivariantization
of a faithful $G$-crossed braided (i.e., $G$-graded central) extension of $\B$. 
Theorem~\ref{Kunneth general}, in particular, characterizes the centers of central $G$-extensions of a symmetric
fusion category $\E$ as $G$-gaugings of $\Z(\E)$ in which  the associated action of $G$ on $\E\subset \Z(\E)$ is trivial.
For $\E=\Vect$, one recovers from this result the classification of twisted group doubles from  \cite[Theorem 4.64]{DGNO1}.
\end{remark}

\begin{example}
\label{recovering Kunneth}
Let $L$ be a finite group  and set $\E =\Rep(L)$  in Theorem~\ref{Kunneth general}.  We recover the K\"unneth formula \eqref{Kunneth formula}
as follows.  In this case, $\Pic(\E)= H^2(L,\,\kk^\times)$  and sequence~\eqref{sequence 1} (with $\mathcal{G}=\uPic(\E)$) splits:
\begin{equation}
\label{UCT}
\Fun(G,\, \uPic(\E)) \cong \Hom(G,\, H^2(L,\,\kk^\times)) \oplus H^2(G,\, \widehat{L}),
\end{equation}
where $\widehat{L} =\Hom(L,\, \kk^\times)$. 
We claim that  the obstruction $\delta$ in \eqref{exact sequence Funsym} vanishes. Indeed, it suffices to check that it vanishes 
on both summands of \eqref{UCT}. Vanishing on the first one  follows for the existence of a Schur covering group $L^*$ of $L$ 
\cite[Section 2.1]{K}, since $\Rep(L^*) $ is a faithful $H^2(L,\,\kk^\times)$-graded extension of $\Rep(L)$.
Vanishing on the second one follows from 
Example~\ref{obstruction for Pic(E)}, since for the Tannakian category $\Rep(L)$ the quadratic homomorphism $q:\widehat{L}\to
\mathbb{Z}/2\mathbb{Z}$  is trivial.

Therefore, the sequence \eqref{exact sequence Funsym for Picard} becomes the following 
short  exact sequence:
\[
0 \to H^3(G,\,\kk^\times) \to \TwoFun(G,\, \uuPic(\E)) \to \Hom(G,\, H^2(L,\,\kk^\times)) \oplus H^2(G,\, \widehat{L}) \to 0.
\]
By the universal coefficient theorem, the last summand is further decomposed as
\[
H^2(G,\, \widehat{L})  \cong \Ext(G/G',\, \widehat{L}) \oplus \Hom(L,\, H^2(G,\, \kk^\times)) \cong \left( \widehat{G}\otimes \widehat{L} \right)
\oplus  \Hom(L,\, H^2(G,\, \kk^\times)),
\]
where $G'$ is the commutator subgroup of $G$.
Combining this with Theorem~\ref{Kunneth general}, we recover all the summands 
of $H^3(G\times L) \cong \Mext(\Rep(G) \bt \Rep(L))$  in the K\"unneth formula \eqref{Kunneth formula}.
\end{example}

\begin{example}
\label{recursive pointed}
We will deal with minimal extensions of {\em pointed} symmetric fusion categories in Section~\ref{Section pointed}.
For now, let us note that Theorem~\ref{Kunneth general} allows calculating of
their  groups of minimal extensions  inductively as follows.
Let $r(\E)$ denote the finite rank of the group $\Inv(\E)$, i.e., the minimal number of its generators. For $r(\E)=1$,
the group $\Mext(\E)$ will be computed in Section~\ref{cyclic 2-group}. When $r(\E)>1$ we can write
\begin{equation}
\E \cong \Rep(\mathbb{Z}_N) \bt \E_1,
\end{equation}
where $r(\E_1) =r(\E)-1$. By Theorem~\ref{Kunneth general},
\[
\Mext(\E) \cong \Mext(\E_1) \oplus \TwoFun(\mathbb{Z}_N,\, \uPic(\E_1)).
\]
The last summand is computed as follows. Since $H^4(\mathbb{Z}_N,\, \kk^\times) =0$
and $H^2(\mathbb{Z}_N,\,\Inv(\E_!)) =\Ext(\mathbb{Z}_N,\,\Inv(\E_1))$,  sequences \eqref{sequence 1} and \eqref{exact sequence Funsym for Picard}  become
\[
0 \to \Ext(\mathbb{Z}_N,\,\Inv(\E_1))  \to  \Fun(\mathbb{Z}_N,\, \uPic(\E_1) ) \to \Hom(\mathbb{Z}_N, \Pic(\E_1))\to 0
\]
and 
\[
0 \to H^3(\mathbb{Z}_N,\,\kk^\times) \to  \TwoFun(\mathbb{Z}_N,\, \uPic(\E_1)) \to  \Fun(\mathbb{Z}_N,\, \uPic(\E_1) )  \to 0.
\]
Thus, the direct complement of $\Mext(\E_1)$ in  $\Mext(\E)$ has a filtration with factors
\begin{equation}
\label{factors of 2Fun}
H^3(\mathbb{Z}_N,\,\kk^\times)  \cong \mathbb{Z}_N,\,\quad \Ext(\mathbb{Z}_N,\,\Inv(\E_1)), \quad \text{and} \quad  \Hom(\mathbb{Z}_N, \Pic(\E_1)).
\end{equation}
We will see in Section~\ref{Z2Z2 example} that the group 
$\TwoFun(\mathbb{Z}_N, \uPic(\E_1))$ is not, in  general, a direct sum of factors \eqref{factors of 2Fun}.
\end{example}

\begin{remark}
\label{E1=sVect}
Let $\E =\Rep(G) \bt \sVect$. Since $\Mext(\sVect)  \cong  \mathbb{Z}_{16}$,  Theorem~\ref{Kunneth general} gives
\begin{equation}
\Mext( \Rep(G) \bt \sVect)  \cong  \mathbb{Z}_{16} \oplus \TwoFun(G,\, \uPic(\sVect)) .
\end{equation}
%

We have $\Inv(\sVect) \cong \mathbb{Z}_2$ and $\Pic(\sVect) \cong \mathbb{Z}_2$, so it follows from exact sequences  \eqref{sequence 1}
and \eqref{exact sequence Funsym for Picard} that the group $\TwoFun(G,\, \uPic(\sVect))$ has  a filtration with factors
\begin{equation}
\label{factors of 2Fun for sVect}
H^3(G,\,\kk^\times),\,\quad  \Ker\left( H^2(G,\,\mathbb{Z}_2)\xrightarrow{\delta} H^4(G,\, \kk^\times) \right),  \quad \text{and} \quad 
 H^1(G, \mathbb{Z}_2)\,
\end{equation}
where $\delta$ is the obstruction map \eqref{cup composition}. Explicit formulas describing the product in this group are given
in \cite{ABK}.
\end{remark}

\begin{remark}
\label{physicists}
Let $\E=\Rep(\tilde{G},\,t)$ be a general symmetric fusion category, where  the group $\tilde{G}$ fits into a  (not necessarily split)
short exact sequence
\[
1\to \mathbb{Z}_2 = \langle t \rangle \to \tilde{G} \to G \to 1.
\]
In this case,  there is a parameterization of $\Mext(\E)$ by torsors over the cohomology
groups  listed in \eqref{factors of 2Fun for sVect}, see \cite[Section VII.C]{ABK}  and \cite[Table I]{BCHM}. These papers 
also contain explicit formulas for products of minimal extensions  (i.e.,  symmetric invertible fermionic topological phases)
in terms of this parameterization.
\end{remark}

\section{The group of minimal extensions of a pointed  symmetric fusion category}
\label{Section pointed}

\subsection{A canonical grading on  a minimal extension}

Let $A$ be a finite group and let $\E=\Rep(A,\,t)$ be a  symmetric fusion category. 



\begin{proposition}
\label{min ext grading}
Let  $\E \hookrightarrow \C$ be a minimal non-degenerate extension. Then $\C$  is faithfully $A$-graded:
\begin{equation}
\label{the grading}
\C =\bigoplus_{x\in A}\, \C_x, \qquad \C_e=\E,
\end{equation} 
where
$\C_x=\{  X \in \C \mid c_{X,V}c_{V,X} = x|_V \ot \id_X \,\, \text{for all} \,\, V \in \E   \}$, $x\in A$.

Conversely, any $A$-graded braided extension of this form is a minimal non-degenerate extension of $\E$. Two such minimal
extensions are equivalent if  and only if they are equivalent as $A$-graded braided extensions of $\E$.
\end{proposition}
\begin{proof}
It follows from \cite[Section 3.4]{DGNO2} that $\E$-module components of $\C$ are parameterized by characters of $K_0(\E)$,
i.e., by elements of $A$. Namely,  the  squared braiding   of a simple object $X\in \C$ with simple objects of $\phi \in \E$,
where $\phi\in \widehat{A} = \Hom(A,\, \kk^\times)$ determines
a character on $\widehat{A}$, i.e., an element $a_X \in A= \widehat{\widehat{A\,}} $ such that
\[
c_{\phi, X} c_{x,\phi} = \phi(a_X)\, \id_{\phi \ot X}.
\]
It follows from the hexagon axioms that the assignment $X \mapsto a_X$ determines  a grading on $\C$.
Since $\C$ is non-degenerate, we must have $\C_x\neq 0$ for all $x\in A$, i.e., the above grading  is faithful.

Conversely, we claim that an $A$-graded braided extension \eqref{the grading} is non-degenerate.
Let $\Z_{sym}(\C)$ denote the symmetric center of $\C$. For $x\neq e$, we have
$\Z_{sym}(\C)\cap \C_x=0$, since there is $V\in \E$ such that $x|_V \neq \id_V$, i.e., $c_{X,V}c_{V,X}  \neq \id_{V \ot X}$ for all non-zero $X\in \C_x$.
If $V\in \Z_{sym}(\C)\cap \E$ is a non-trivial representation of $A$, then there is $x\in A$ such that $x|_V \neq \id_V$, i.e., $V$ does not centralize $\C_x$.
Hence, $\Z_{sym}(\C)=\Vect$, i.e., $\C$ is non-degenerate. It is a minimal extension since $\FPdim(\C)=\FPdim(\E)^2$.

An equivalence of minimal extensions preserves the squared braiding and restricts to the identity on $\E$. Therefore it must preserve the
grading  \eqref{the grading}.
\end{proof}

Recall  \cite{GN} that a fusion category is {\em nilpotent} if it is obtained from $\Vect$ by a sequence of graded extensions.

\begin{corollary}
\label{nilpotent}
Let $\E$ be a pointed symmetric fusion category and let $p_1,\dots, p_n$ be distinct primes dividing $\FPdim(\E)$.
Let $\E=\E_{1} \boxtimes \cdots \bt \E_{n}$ be the Sylow decomposition of $\E$,
where  $\FPdim(\E_i)$ is a power of $p_i,\, i=1,\dots,n$. Then any minimal extension of $\E$ is nilpotent and
\begin{equation}
\label{nilpotent Mext}
\Mext(\E) = \Mext(\E_{1}) \times \cdots \times \Mext(\E_{n}).
\end{equation}
\end{corollary}
\begin{proof}
By Proposition~\ref{min ext grading}, a minimal extension of $\E$ is a graded extension of a pointed fusion category, so
it is nilpotent of nilpotency class at most $2$. It is shown in \cite[Theorem 6.12]{DGNO1} that
any nilpotent braided fusion category  $\C$ admits a Sylow decomposition $\C = \C_{1} \times \cdots \times \C_{n}$. So if  $\E \hookrightarrow \C$
is a minimal extension then $\E_{i} \hookrightarrow \C_{i}$ is a  minimal extension for all $i=1,\dots n$.
This implies the statement.
\end{proof}

\begin{remark}
\label{only p=2}
Minimal extensions of Tannakian fusion categories were classified in  \cite{LKW1} where it was shown  that $\Mext(\Rep(G))=H^3(G,\,\kk^\times)$ 
for any finite group $G$. In view of Corollary~\ref{nilpotent},  to classify extension of pointed symmetric fusion categories it remains
to determine $\Mext(\Rep(A,\, t))$, where $A$ is an Abelian $2$-group.
\end{remark}


\begin{remark}
\label{special extensions}
It follows from the description of the homogeneous components of the grading \eqref{the grading} that
the groupoid $\uMext(\Rep(A,\,t))$ is equivalent  to the groupoid of braided monoidal $2$-functors $F:A\to \uuPicbr(\Rep(A,\,t))$ such that 
the composition of group homomorphisms
\begin{equation}
\label{composition idA}
A \xrightarrow{\pi_0(F)} \Pic_{br}(\Rep(A,\,t)) =\Pic(\Rep(A,\,t)) \times A  \xrightarrow{p_A} A,
\end{equation}
equals $\id_A$. Here $p_A$ denotes the projection on $A$.
%
\end{remark}

\subsection{A canonical filtration of $\Mext(\E)$}


We say that a minimal extension $\E \hookrightarrow \C$ is {\em integral} if $\C$ is an integral
fusion category, i.e., $\FPdim(X)$ is an integer for all objects $X\in \C$.  Such  extensions are characterized
by the following property: the image of the corresponding homomorphism composition
\begin{equation}
\label{composition A to pi0} 
A  \to \Picbr(\E)\to \Pic(\E)
\end{equation}
lies in $\Pic_{int}(\E)$, see \eqref{Picint = H2 Gtk}.

An integral minimal extension $\E \hookrightarrow \C$  is {\em pointed} if  $\C$ is a pointed category. 
Equivalently,  this extension  is {\em quasi-trivial}
in the sense of \cite[Section 8.7]{DN2}, i.e., the homomorphism \eqref{composition A to pi0}  is trivial.
In this case, the  homomorphism $A\to  \Picbr(\E)$ is identified with the
identity map  $A \to \Aut_\ot(\id_\E) =A$.  The braided monoidal functor 
\begin{equation}
\label{A to Picbr}
A \to \uPicbr(\E) 
\end{equation}
is determined by an element $L\in H^2_{ab}(A,\, \widehat{A}) = \Ext(A,\, \widehat{A})$.  

Here and below
$H^*_{ab}(A,M) = H^{*+1}(K(A,2),M)$ denotes the abelian Eilenberg-Mac Lane cohomology group of $A$
with coefficients in $M$  \cite{EM1}.  A description of low dimensional abelian cohomology groups $H^n_{ab}(A,\, M),\, n \leq 4$ 
can be found, e.g.,  in \cite[Section 2.1]{DN2}, where the term ``braided cohomology" was used.

Finally, a pointed minimal extension $\E \hookrightarrow \C$ is {\em trivial}  if the monoidal 
functor \eqref{A to Picbr} is trivial. Such extensions are easy to describe explicitly as follows.
We have $\E=\Rep(A,\,t) =\C(\widehat{A},\, t)$, where $t$ is viewed as a quadratic character on $\widehat{A}$.
For any quadratic form $q:A \to \kk^\times$ define a quadratic form $h_q: A \times \widehat{A} \to \kk^\times$ by
\begin{equation}
h_q(a,\, \phi) = \langle at,\,  \phi \rangle \, q(a),\qquad a\in A,\, \phi \in \widehat{A}.
\end{equation}
It is easy to see that this form is non-degenerate and 
\begin{equation}
\label{typical trivial}
\E =\C(\widehat{A},\, t)   \hookrightarrow \M_q:=\C(A \times \widehat{A} ,\, h_q).
\end{equation}
is a typical trivial minimal extension of $\E$.

\begin{lemma}
\label{closed}
The set of trivial (respectively, pointed,  integral)  minimal extensions of $\E$ is  closed under the tensor product.
\end{lemma}
\begin{proof}
The statement about trivial extensions follows from their explicit description \eqref{typical trivial}. Indeed, 
one can directly check that the assignment
\begin{equation}
\label{H3 to Mext}
H^3_{ab}(A,\, \kk^\times) \to \Mext(\Rep(A,\,t)) : q \mapsto \M_q
\end{equation}
is a group homomorphism.

Since $\E$ is pointed, the tensor product of its minimal extensions is obtained by taking a de-equivariantization  with respect
to the diagonal Tannakian subcategory in $\E \bt \E$.  Clearly, a de-equivariantization of a pointed (respectively, integral)
fusion category is pointed (respectively, integral). This proves the statement about pointed and integral extensions. 
\end{proof}

\begin{remark}
Homomorphism  \eqref{H3 to Mext} is, in general, not injective.
\end{remark}

Thus, we have a filtration
\begin{equation}
\label{cat filtration}
\uMext_{triv}(\E) \subset \uMext_{pt}(\E) \subset  \uMext_{int}(\E) \subset \uMext(\E),
\end{equation}
where $\uMext_{triv}(\E),\,\uMext_{pt}(\E),\,  \uMext_{int}(\E)$  denote the  categorical groups
of trivial, pointed, and integral minimal extensions of $\E$. There is a corresponding filtration
of Abelian groups
\begin{equation}
\label{filtration}
\Mext_{triv}(\E) \subset \Mext_{pt}(\E) \subset  \Mext_{int}(\E) \subset \Mext(\E),
\end{equation}
which we are going to study next. Our goal is to determine  the factors of this filtration.

\subsection{Trivial minimal extensions}

Recall that the third abelian cohomology group $H^3_{ab}(A,\,\kk^\times)$  is isomorphic to the group 
$\Quad(A,\,\kk^\times)$ of quadratic forms on $A$ with values in $\kk^\times$.

\begin{proposition}
\label{Mext triv}
$\Mext_{triv}(\E) \cong \Coker \left( H^1(A,\, \widehat{A}) \xrightarrow{\kappa^t} H^3_{ab}(A,\,\kk^\times) \right)$, where
\begin{equation}
\label{kappa}
\kappa^t: H^1(A,\, \widehat{A}) \to \Quad(A,\,\kk^\times) : Z \mapsto q_Z,\qquad   q_Z(x)= \langle xt,\, Z(x) \rangle,\, x\in A.
\end{equation}
\end{proposition}
\begin{proof}
A trivial extension of $\E=\Rep(A,\,t)$ is obtained by  deforming the structure constraints of the identity extension
\[
\E \hookrightarrow \Z(\E) = \bigoplus_{a\in A}\,  \Z(\E)_a
\]
by means of an abelian $3$-cocycle $(\omega,\, c) \in Z^3_{ab}(A,\,\kk^\times)$. This is a special case of a 
{\em zesting} procedure studied in \cite{DGPRZ}. Namely, let $a_{W,X,Y}$ and $c_{X,Y}$ denote 
the associativity and braiding isomorphismss in $\Z(\E)$. The deformed extension $\Z(\E)^{ (\omega,\, c) }$ coincides with
$\Z(\E)$ as an abelian category and has the same tensor product, while its associativity and braiding isomorphisms are given by
\begin{eqnarray*}
\tilde{a}_{W,X,Y}  &=&  \omega( \deg(X), \deg(Y),\deg(Z))\, a_{W,X,Y}, \\
\tilde{c}_{X,Y}  &=&  \omega( \deg(X),\deg(Y))\, c_{X,Y}, 
\end{eqnarray*}
for all homogeneous objects $W,X,Y$.

Isomorphisms between trivial braided extensions were classified in \cite[Section 8.7]{DN2}. In particular, formula \eqref{kappa}
follows from \cite[formula (8.52)]{DN2}, since the self braiding $c_{X,X}$ of a simple object $X \in \Rep(A,\,t)$ is given by the evaluation 
$\langle X,\, t \rangle$.
\end{proof}

\begin{proposition}
\label{explicit tri}
Let $A(2) = \mathbb{Z}_{2^{n_1}} \times \cdots  \times \mathbb{Z}_{2^{n_r}}$ be the Sylow $2$-subgroup of $A$.  Then 
\begin{equation}
\label{mext triv formula}
\Mext_{triv}(\E) \cong 
\begin{cases}
\mathbb{Z}_4 \times \mathbb{Z}_2^{r-1} & \text{if $\langle t \rangle$ is a direct summand of $A$}, \\
\mathbb{Z}_2^{r} & \text{otherwise}.
\end{cases}
\end{equation}
\end{proposition}
\begin{proof}
We write $A = A(2)\times A(\text{odd})$ and note that the map $\kappa^t$ from \eqref{kappa} respects this decomposition and is
an isomorphism on $A(\text{odd})$. So we may assume that $A$ is a $2$-group. 
Isomorphism \eqref{mext triv formula} is easy to check when $A= \mathbb{Z}_{2^{n}} $ is cyclic. In this case, we have
\[
H^1(\mathbb{Z}_{2^{n}},\,\mathbb{Z}_{2^{n}})= \mathbb{Z}_{2^{n}}, \quad \Quad(\mathbb{Z}_{2^{n}}, \, \kk^\times) \cong \mathbb{Z}_{2^{n+1}},
\]
and the homomorphism $\kappa^t:\mathbb{Z}_{2^{n}} \to \mathbb{Z}_{2^{n+1}}$ is injective if $n>1$ and is zero for $n=1$.

Let $A = \langle x_1 \rangle \times \cdots \times \langle x_r \rangle$, where $|x_i|=2^{n_i}$.  We may assume that $t\in  \langle x_1 \rangle$.
A quadratic form $q:A \to \kk^\times$ is uniquely determined by the values 
\[
q(x_i),\, i=1,\dots, r\quad \text{and} \quad  b(x_j,\, x_l):= \frac{q(x_j x_l)}{q(x_j) q(x_l)},\, 1\leq j <l\leq r.
\]
Here $q(x_i)$ is a  $2^{n_i+1}$th root of unity  and $b(x_j,\, x_l)$ is a  $\min\{ 2^{n_j},\, 2^{n_l}\}$th root of unity.
Any choice of such roots of unity will give a quadratic form.
Let us identify $Z\in H^1(A,\, \widehat{A})$  with a bilinear form on $A$.
The symmetric bilinear form associated to $q_Z$ is 
\[
b_Z(x,y) : = Z(x,\,y)\, Z(y,\,x),\qquad x,y \in A,
\]
which can realize all possible values of  $b(x_j,\, x_l)$.  On the other hand, $q_Z(x_i) = Z(x_i, t x_i )$ can be any  $2^{n_i}$th root of unity
if $i>1$ and $q_Z(x_1) = Z(x_1, t ) Z(x_1, x_1)$. The latter can be any $2^{n_1}$th  root of unity if $t\neq x_1$ (i.e., $n_1>1$)  and 
iequals $1$ otherwise. From this, the cokernel of \eqref{kappa} is easily determined.
\end{proof}

\subsection{Pointed minimal extensions}

The group $Mext_{pt}(\Rep(A,\,t))$ can be computed using the classification of quasi-trivial
graded extensions from \cite[Section 8.7]{DN2}.  For an abelian group $A$ 
we identify $ \Ext(A,\, \widehat{A})$ with $H^2_{ab}(A,\, \widehat{A})$
There is a natural involution
\[
\eps: H^2_{ab}(A,\, \widehat{A})  \to H^2_{ab}(A,\, \widehat{A}) 
\]
that sends the class of an abelian extension 
\begin{equation}
\label{0ACA0}
0\to \widehat{A}\to C \to  A \to 0
\end{equation}
to the class of the dual extension obtained by applying the functor $\Hom(-,\, \kk^\times)$ to \eqref{0ACA0}:
\begin{equation}
\label{dual 0ACA0}
0\to \widehat{A}\to \widehat{C} \to  \widehat{\widehat{A\,}} =A \to  0.
\end{equation}
 This $\eps$ was explicitly described in \cite[Section 8]{MN}.
Namely, let $L=\{L_{x,y}\}_{x,y\in A}$  be a normalized $2$-cocycle in $H^2_{ab}(A,\, \widehat{A})$ corresponding
to the extension \eqref{0ACA0}.  For any $z\in A$ there is a normalized $1$-cochain $a_z \in C^1(A,\, \kk^\times)$ such that
\begin{equation}
\label{coboundary L}
L_{x,y}(z) = \frac{a_z(x) a_z(y)} {a_z(xy)},\qquad x,y,z\in A,
\end{equation}
and  $\eps(L)$ is determined by
\begin{equation}
\label{coboundary epsL}
\eps(L)_{x,y}(z) = \frac{a_x(z) a_y(z)} {a_{xy}(z)},\qquad x,y,z\in A.
\end{equation}
Let us denote $H^2_{ab}(A,\, \widehat{A})^\eps =\{ L \in H^2_{ab}(A,\, \widehat{A}) \mid \eps(L)=L \}$. 

By Proposition~\ref{min ext grading}, an extension $\C\in \Mext_{pt}(\E)$ defines a quasi-trivial $A$-graded braided extension of $\E$.
The  corresponding  braided  monoidal functor $F_\C: A \to \uPicbr(\E)$ is completely determined by 
an abelian $2$-cocycle $L_\C \in H^2_{ab}(A,\, \widehat{A}) $ defining the monoidal structure of $F_\C$.  This $L_\C$
is precisely the $2$-cocycle corresponding to the central extension  $0\to \widehat{A}\to \Inv(\C)  \to  A \to 0$.

Thus, there is a homomorphism 
\begin{equation}
\label{Mext to H2ab}
\lambda: \Mext_{pt}(\E)\to   H^2_{ab}(A,\, \widehat{A})  : \C \mapsto L_\C
\end{equation}
whose kernel is $\Mext_{triv}(\E)$.  The image of $\lambda$ consists of all $L \in H^2_{ab}(A,\, \widehat{A})$ such that
the corresponding braided monoidal functor $A \to \uPicbr(\E)$  admits an extension to a  braided monoidal $2$-functor.
By \cite[Section 8.7]{DN2} this image is the kernel of the Pontryagin-Whitehead homomorphism 
\[
PW^2: H^2_{ab}(A,\, \widehat{A}) \to H^4_{ab}(A,\, \kk^\times)
\]
whose components are given by formulas  \cite[(8.53)-(8.55)]{DN2}:
\begin{equation}
\label{PW2}
PW^2(L)(x,y,z,w)= c_{L_{x,y}, L_{z,w}},\qquad PW^2(L)(x,y|,z)=1,\qquad PW^2(L)(x,|y,z)= L_{y,z}(x),
\end{equation}
for all $x,y,z,w\in A$.  Thus, there is an exact sequence of group homomorphisms
\begin{equation}
\label{SE for Mextpt}
0 \to \Mext_{triv}(\E) \to \Mext_{pt}(\E)\xrightarrow{\lambda} H^2_{ab}(A,\, \widehat{A}) \xrightarrow{PW^2} H^4_{ab}(A,\, \kk^\times).
\end{equation}

For an Abelian group $A$ let us denote  $A_2=\{ x\in A \mid x^2 = e\}$.

\begin{proposition}
\label{Mext pt}
There is a group isomorphism
\[
\Mext_{pt}(\E)/ \Mext_{triv}(\E) \cong \Ker \left(  H^2_{ab}(A,\, \widehat{A})^\eps \xrightarrow{\theta^t}  \Hom(A_2/ \langle t \rangle,\,\kk^\times) \right),
\]
where
\begin{equation}
\label{theta}
\theta^t:  H^2_{ab}(A,\, \widehat{A})^\eps \to \Hom(A_2/ \langle t \rangle,\,\kk^\times) :  L \mapsto  \theta^t_L,\qquad \theta^t_L(x) =L_{x,x}(xt),\, x\in A.
\end{equation}
\end{proposition}
\begin{proof}
Given the exact sequence \eqref{SE for Mextpt}, all we need to show is that the kernels of $\theta^t$ and $PW^2$ coincide. 
It was shown in \cite{EM2}  (see also \cite[(2.19)]{DN2})  that  there is an exact sequence
\[
0 \to \Hom(A_2,\, \kk^\times)\to H_{ab}^4(A,\,\kk^\times)  \xrightarrow{h_A} H^2_{ab}(A,\, \widehat{A}).
\]
It  follows from the construction described at the end of  \cite[Section 2.1]{DN2} and  formulas \eqref{coboundary L}, \eqref{coboundary epsL}
that 
\begin{equation}
\label{h PW2}
h_A(PW^2(L))= L\,\eps(L)^{-1},\qquad  L \in H^2_{ab}(A,\, \widehat{A}).
\end{equation}
%
%
There is a canonical isomorphism
\begin{equation}
\label{xxxx}
\iota_A: \Ker(h_A) \xrightarrow{\sim} \Hom(A_2,\, \kk^\times):  \alpha \mapsto \iota_A(\alpha),
\end{equation}
defined by 
\begin{equation}
\label{xxxx formula}
\iota_A(\alpha)(x) = \frac{\alpha(x,x,x,x) \alpha(x|x,x)}{\alpha(x,x|x)}, \qquad x\in A_2.
\end{equation}
Combining formulas  \eqref{PW2}  and \eqref{xxxx formula} we obtain
\begin{equation}
\label{Iota PW2}
\iota_A(PW^2(L))(x)=  L_{x,x}(xt),\quad x\in A.
\end{equation}
Thus, $PW^2(L)=0$ if and only if  $L\in H^2_{ab}(A,\, \widehat{A})^\eps$ and $L_{x,x}(xt) =1$ for all $x\in A$,
i.e., $L \in \Ker(\theta^t)$.
%
Since the right hand side of \eqref{Iota PW2}  vanishes on $t$, we conclude that $PW^2$ descends to a homomorphism 
\[
\theta^t: H^2_{ab}(A,\, \widehat{A})^\eps \to  \Hom(A_2/ \langle t \rangle,\,\kk^\times)
\]
defined in \eqref{theta}.  
\end{proof}

\begin{remark}
The group  $H^2_{ab}(A,\, \widehat{A})$ is (non-canonically) isomorphic to $\Hom(\widehat{A}\ot \widehat{A},\,\kk^\times)$.
It was explained in \cite[Lemma 8.2]{MN} that, upon this isomorphism, $\eps$ is identified with the transposition map, so
\[
H^2_{ab}(A,\, \widehat{A})^\eps \cong \Sym^2(\widehat{A}).
\]
In particular, if $A$ is cyclic, then $\eps$ is the identity map.
\end{remark}

\begin{proposition}
\label{theta surj}
Homomorphism \eqref{theta} is surjective.
\end{proposition}
\begin{proof}
We may assume that $A= C_1\times \cdots \times C_r$, where each $C_i$ is a cyclic $2$-group  and $t\in C_1$.  It suffices to check that 
for $C_i =\langle x \mid x^{2N_i} =e \rangle,\, i=2,\dots, r$,  the homomorphism
\begin{equation}
\label{H2 to pm1}
H^2_{ab}(C_i,\, \widehat{C_i}) \to \Hom(\mathbb{Z}_2,\, \kk^\times) \cong \{\pm 1 \}  : L \mapsto  L_{x^{N_i},\, x^{N_i}}(x^{N_i}) 
\end{equation}
is surjective. To see that, 
let $\xi$ be a primitive $2N_i$-th root of $1$ in $k$ and let  $L$ be a generator of $H^2_{ab}(C_i,\, \widehat{C_i}) \cong \mathbb{Z}_{2N_i}$ 
explicitly defined by
\[
L_{x^k, x^l}(x^m)  =
\begin{cases}
1 & \text{if  $k+l < 2N_i$ }\\
\xi^m &  \text{if  $k+l \geq  2N_i$},\qquad  m=1,\dots, 2N_i.
\end{cases}
\]
Then  \eqref{H2 to pm1} sends $L$ to $-1$, as required.
\end{proof}

\subsection{Integral minimal extensions}

We continue to denote $\E=\Rep(A ,\,t)$. 

\begin{lemma}
\label{index=lambda3}
We have 
\begin{equation}
\label{int pt index equality} 
[\Mext_{int}(\E) : \Mext_{pt}(\E)] = |{\wedge}^3 A|.
\end{equation}
\end{lemma}
\begin{proof}
We will prove this by induction on $r(\E)$, the finite rank of the Abelian group $\Inv(\E)$. When $r(\E) =1$, i.e., when $A$ is cyclic,
we have $\Pic_{int}(\E)=0$ and, hence,  all  minimal extensions of $\E$ are pointed, so both sides of \eqref{int pt index equality} are equal to $0$.

As in Example~\ref{recursive pointed}, any $\E=\Rep(A,\,t)$ with $r(\E)=r+1,\, r>1,$ can be written as $\E =\Rep(\mathbb{Z}_N) \bt \E_1$, where $r(\E_1)=r$
and $\E=\Rep(A_1,\,t)$.  We have
\begin{equation}
\label{int pt index =}
[\Mext_{int}(\E) : \Mext_{pt}(\E)]  =
[\Mext_{int}(\E_1) : \Mext_{pt}(\E_1)]  \times   \frac{|\Mext_{int}(\E)|}{|\Mext_{int}(\E_1)|}  \times \frac{|\Mext_{pt}(\E_1)|}{|\Mext_{pt}(\E)|}.
\end {equation}
The first factor in \eqref{int pt index =} is equal to $|{\wedge}^3 A_1|$ by the inductive assumption.  
By Example~\ref{recursive pointed}, the second factor 
is equal to
\begin{equation*}
 |H^3(\mathbb{Z}_N,\, \kk^\times)|  \times |\Ext(\mathbb{Z}_N,\, \widehat{A}_1)| \times |\Hom(\mathbb{Z}_N,\, {\wedge}^2 A_1)| = 
 |\mathbb{Z}_N| \times |\mathbb{Z}_N \otimes A_1| \times |\mathbb{Z}_N \otimes {\wedge}^2 A_1| .
\end{equation*}
Using Propositions~\ref{explicit tri}, \ref{Mext pt}, and \ref{theta surj}, we compute the last factor in \eqref{int pt index =}  as
\begin{gather*}
\frac{|\Mext_{triv}(\E_1)|}{|\Mext_{triv}(\E)|} \times \frac{[\Mext_{pt}(\E_1):\Mext_{triv}(\E_1) ] }{[Mext_{pt}(\E): \Mext_{triv}(\E)]}
= 2 \times \frac{| H^2_{ab}(A,\,\widehat{A})^\eps|}{2|H^2_{ab}(A_1,\,\widehat{A}_1)^\eps|}\\
 =  \frac{ |\Sym^2(A_1)|} {|\Sym^2(A)|} = \frac{1}{|\mathbb{Z}_N| \times |\mathbb{Z}_N \otimes A_1|}.
\end{gather*}
Substituting these quantities into  \eqref{int pt index =} we obtain
\begin{equation}
\label{int pt index = substituted}
[\Mext_{int}(\E) : \Mext_{pt}(\E)]  = |\mathbb{Z}_N| \times |\mathbb{Z}_N \otimes A_1| \times |{\wedge}^3 A_1| = |{\wedge}^3 A|,
\end {equation}
as required. 
\end{proof}

Recall from \eqref{the grading} that a minimal extension  $\E \hookrightarrow \C$ admits a canonical 
faithful $A$-grading
\begin{equation}
\label{grading recalled}
\C =\bigoplus_{x\in A} \C_x,\qquad \C_e=\E.
\end{equation}
When this extension is integral, its components are of the form $\C_x =\Rep(\kk_{\mu_x}[G])$, where
\begin{equation}
\label{mu}
A \to H^2(A,\,t,\,\kk^\times) =\Pic_{int}(\E) : x\mapsto \mu_x
\end{equation}
is a group homomorphism.  It follows from the description of the group $\Picbr(\E)$ in Section~\ref{sect Picard groups} and
Remark~\ref{special extensions} that the corresponding homomorphism $A\to \Picbr(\E)$ is
\begin{equation}
\label{mu pair}
A \to H^2(A,\,t,\,\kk^\times) \times A : x\mapsto (\mu_x,\, x).
\end{equation}

Since \eqref{mu pair}  comes from a braided monoidal $2$-functor
$A \to \uuPicbr(\E)$,  it  must  satisfy
\begin{equation}
\label{Q mux x}
Q_\E(\mu_x,\,x) =1,\quad \text{for all $x\in A$},
\end{equation}
where $Q_\E: H^2(A,\,t,\,\kk^\times) \times A \to \widehat{A}$ is the quadratic form \eqref{Qint}.

Define a map $\tau_\C: A^3 \to \kk^\times$  by
\begin{equation}
\label{trilinear}
\tau_\C(x,\,y,\,z) =   \left(-1 \right)^{\xi_{\mu_x}(y)\xi_{\mu_x}(z)}\,  \frac{\mu_x(y,z)}{\mu_x(z,y)},
\end{equation}
where the bilinear map $\xi_\mu: H^2(A,\,t,\,\kk^\times)  \times A\to \mathbb{Z}/2\mathbb{Z}$, was introduced in \eqref{form xi}.

\begin{proposition}
\label{trilinear form}
The map \eqref{trilinear} is a trilinear form on $A$ satisfying
\begin{equation}
\label{t-alternating}
\text{ $\tau_\C(x,\,tx,\,y) = 1$ and   $\tau_\C(x,\, ty,\, y) = 1$.} 
\end{equation}
for all $x,y \in A$.
\end{proposition}
\begin{proof}
The linearity of $\tau_\C$ in the second and third arguments is clear since for each $\mu\in H^2(A,\, \kk^\times)$
the map 
\[
A^2 \to \kk^\times: (x,y) \mapsto \frac{\mu(x,\,y)}{\mu(y,\,x)}
\] 
is an alternating bilinear form. To check the linearity in the first argument, we compute, using the definition 
 of the product $*$ from \eqref{product *}: 
\begin{eqnarray*}
\tau_\C(wx,\, y, z)
&=&  \left(-1 \right)^{\xi_{\mu_{wx}}(y)\, \xi_{\mu_{wx}(z)}}\,  \frac{\mu_{wx}(y,z)}{\mu_{wx}(z,y)} \\
&=&   \left(-1 \right)^{\xi_{\mu_{w} * \mu_x}(y) \,\xi_{\mu_{w} *\mu_x(z)}}\,  \frac{(\mu_{w} *\mu_x)(y,z)}{(\mu_{w} *\mu_x)(z,y)}  \\
&=& \left(-1 \right)^{\xi_{\mu_{w} * \mu_x}(y) \,\xi_{\mu_{w} *\mu_x(z)}}\,   
\left(-1 \right)^{  \xi_{\mu_{x}}(y)   \xi_{\mu_{w}}(z) +   \xi_{\mu_{w}}(y)   \xi_{\mu_{x}}(z) }
        \frac{\mu_{w}(y,z)}{\mu_{w}(z,y)}  \, \frac{\mu_{x}(y,z)}{\mu_{x}(z,y)}  \\
&=&   \left(-1 \right)^{\xi_{\mu_x}(y)\xi_{\mu_x}(z) + \xi_{\mu_w}(y)\xi_{\mu_w}(z)}  \,
           \frac{\mu_{w}(y,z)}{\mu_{w}(z,y)}  \, \frac{\mu_{x}(y,z)}{\mu_{x}(z,y)}  \\
&=&  \tau_\C(w,\, y, z) \tau_\C(x,\, y, z),
\end{eqnarray*}
for all  $x,w,y,z\in A$, where we used that $\xi_{\mu *\nu} = \xi_{\mu}  +\xi_{\nu}$ for all $\mu,\nu \in H^2(A,\,t,\,\kk^\times)$.

The condition \eqref{Q mux x} along with the formula \eqref{Qint}
imply that $\xi_{\mu_x}(x) =0$, i.e., $\frac{\mu_x(x,\,t)}{\mu_x(t,\,x)}=1$, and 
\[
1 = Q_\E(\mu_x,\,x)(y) = \frac{\mu_x(xt,\,y)}{\mu_x(y,\,xt)} =   \left(-1 \right)^{\xi_{\mu_x}(xt)\xi_{\mu_x}(y) } \, \frac{\mu_x(xt,\,y)}{\mu_x(y,\,xt)}  =\tau_\C(x,\,xt,\,y),
\]
 for all $x,y\in A$, which is the first identity in \eqref{t-alternating}. Finally,
 \[
 \tau_\C(x,\,yt,\,y) = \frac{\mu_x(yt,\,y)}{\mu_x(y,\,yt)} \, \left(-1 \right)^{\xi_{\mu_x}(yt)\xi_{\mu_x}(y) } 
 = \frac{\mu_x(t,\,y)}{\mu_x(y,\,t)}\,  \left(-1 \right)^{\xi_{\mu_x}(y)^2 }  =  \left(-1 \right)^{\xi_{\mu_x}(y)+ \xi_{\mu_x}(y)^2 }  =1,
 \]
which is the second identity in \eqref{t-alternating}.
\end{proof}


%

\begin{proposition} 
\label{Mext int}
There is a (non-canonical) group isomorphism
\begin{equation}
\label{embedding into wedge3}
\Mext_{int}(\E) /  \Mext_{pt}(\E) \cong \Hom({\wedge}^3 A ,\, \kk^\times)
\end{equation}
\end{proposition}
\begin{proof}
The map \eqref{trilinear} defines  a group homomorphism
\begin{equation}
\label{p to wedge}
\tau: \Mext_{int}(\E) \to \Hom(A^{\ot 3},\, \kk^\times) : \C \mapsto \tau_\C.
\end{equation}
A minimal extension $\C$ is pointed if and only if the corresponding homomorphism \eqref{mu} is trivial, hence
pointed extensions belong  to the kernel of $\tau$.  Conversely,  if $\tau_\C=1$ then 
\[
1= \tau_\C(x,\,y,\,y) = \left(-1 \right)^{\xi_{\mu_x}(y)},
\]
so that $\xi_{\mu_x}(y)=0$ for all $x,y\in A$.
This implies that $\tau_\C(x,\,y,\,z) =\frac{\mu_x(y,\,z)}{\mu_x(z,\,y)}=1$ for all $x,y,z\in A$, so that $\mu_x=0$
in $H^2(A,\,t,\,\kk^\times)$
and $\C$ is pointed.

Thus, $\Mext_{int}(\E) / \Mext_{pt}(\E)$ is isomorphic to the image of $\tau$ in $\Hom(A^{\ot 3},\, \kk^\times)$.
We can choose a presentation $A =\langle e_1 \rangle \times \cdots  \times \langle e_r \rangle$ such that we also have
$A =\langle te_1  \rangle \times \cdots  \times \langle te_r \rangle$.  By Proposition~\ref{trilinear form}, the trilinear form
$\tau_\C$ is completely determined by its values $\tau_\C(e_i,\, te_j,\, e_k)$ when $i,j,k$ are distinct.  Hence, 
the group of such forms is embedded into  $\Hom({\wedge}^3 A ,\, \kk^\times)$. By Lemma~\ref{index=lambda3},
this embedding must be an isomorphism.
\end{proof}

\begin{remark}
\label{T remark}
For a Tannakian category $\E=\Rep(A)$, Proposition~\ref{Mext int} implies  a well-known fact that the
alternator homomorphism $\alt: H^3(A,\, \kk^\times) \to \Hom({\wedge}^3 A,\, \kk^\times)$ defined by
\begin{equation}
\label{alt}
\alt(\omega)(x,\,y,\,z)   =\prod_{\sigma\in S_3}\, \omega(\sigma(x),\,\sigma(y),\, \sigma(z))^{\text{sign}(\sigma)}
\end{equation}
for  all  $x,y,z\in A$, is  surjective.
Indeed,  a typical minimal extension of $\Rep(A)$ is  $\Rep(A) \hookrightarrow \Z(\Vec_A^\omega)$ for
some $\omega \in H^3(A,\, \kk^\times)$. In this case,  the $2$-cocycles $\mu_x,\,x\in A,$ 
in \eqref{mu} are given by
\begin{equation}
\label{Schur multiplier gamma}
\mu_x(y,z) =\frac{ \omega(x,\, y,\, z) \omega(y,\, z,\, x)}{\omega(y,\, x,\, z)},\qquad y,z\in G,
\end{equation}
and 
\[
\tau_{\Z(\Vec_A^\omega)}(x,\,y,\,z) = \frac{\mu_x(y,\,z) }{\mu_x(z,\,y)} = \alt(\omega)(x,\,y,\,z).
\]
\end{remark}

\begin{remark}
\label{S remark}
An integral minimal extension of a pointed super-Tannakian category  $\E=\Rep(A,\,t)$ of central charge $1$ must be
equivalent to $\E \hookrightarrow  \Z(\Vec_G^\omega)$ for some  $\omega$ in $H^3(G,\, \kk^\times)$ \cite{DGNO1}.
But it is not always possible to choose a group $G$ to be Abelian. 

To see that twisted Drinfeld doubles of $A$ are not sufficient, note 
that $ \Z(\Vec_{A}^{\omega})$, where  $\omega \in H^3(A,\, \kk^\times),$ 
contains $\E$ as a fusion category only if \eqref{alt} factors through 
\[
\Hom({\wedge}^3 (A/\langle t \rangle),\, \kk^\times) \to \Hom({\wedge}^3 A,\, \kk^\times). 
\]
Indeed, suppose there is an embedding $\E \hookrightarrow \Z(\Vec_{A}^{\omega})$. Let $\mathcal{T}=  \Rep(A/\langle t \rangle) \subset \E$.
Then $\mathcal{T}'$ contains pointed fusion subcategories $\E$ and $\Rep(A)$.  So $\mathcal{T}'$ must be pointed. But this means that the $t$-component
of the canonical braiding \eqref{the grading} is trivial, so $\mu_t=1$ and $\alt(\omega)(t,\,-,\,-)=1$. 

An example of a minimal extension of $\E$ involving the twisted Drinfeld double of a non-Abelian group  can be constructed as follows.
Let $\mathcal{I}$ be an Ising category.  Consider $\Z(\Z(\mathcal{I})) = \mathcal{I} \bt \mathcal{I} \bt \mathcal{I} ^\rev \bt \mathcal{I}^\rev$.
Let $\C$ be the de-equivariantization of the maximal integral subcategory $\Z(\Z(\mathcal{I}))_{int}$ by its symmetric center (the latter is equivalent
to $\Rep(\mathbb{Z}_2)$). Then $\C$ is a minimal extension of $\C_{pt} \cong \Rep(\mathbb{Z}_2 \times \mathbb{Z}_2 \times \mathbb{Z}_2^f)$.
So $\C$ does not contain any pointed Lagrangian Tannakian subcategories and so is not equivalent to the center of a pointed fusion category.
\end{remark}

\subsection{General minimal extensions}


\begin{proposition}
\label{Mext gen}
\begin{enumerate}
\item[(a)] Let $\E$ be a Tannakian or non-split super-Tannakian category. Then all minimal extensions
of $\E$ are integral, i.e.,  $\Mext(\E) =  \Mext_{int}(\E) $.
\item[(b)] Let $\E=\Rep(A,\,t)$ be a split super-Tannakian category. Then 
\[
\Mext(\E) / \Mext_{int}(\E) \cong \Hom(A, \, \mathbb{Z}_2).
\]
\end{enumerate}
\end{proposition}
\begin{proof}
Part (a) is clear since in this case $\Pic(\E) = \Pic_{int}(\E)$ and so all components of the grading \eqref{the grading}
are integral. 

For the part (b),  let $A=A_0\times \langle t \rangle$ so that $\E =\Rep(A_0) \bt \sVect$.  By Theorem~\ref{Kunneth general},
\[
\Mext(\E) = \Mext(\sVect) \times \Exctr(A_0,\, \sVect).
\]
Combining this with homomorphisms $\Mext(\sVect)  \to \Pic(\sVect)$ and $\Exctr(A_0,\, \sVect) \to
\Hom(A_0, \Pic(\sVect))$ and using that $\Pic(\sVect) \cong \mathbb{Z}_2$, 
we obtain a group homomorphism $A\to \mathbb{Z}_2$. So we have a homomorphism
\begin{equation}
\label{phi}
\Mext_{int}(\E) \to \Hom(A, \, \mathbb{Z}_2).
\end{equation}
Equivalently, this can also be described as follows. Let $F_\C:A\to \uuPicbr(\E)$ be a braided monoidal $2$-functor corresponding
to a minimal extension $\E\hookrightarrow \C$.  The image of this extension under  \eqref{phi} is
\begin{equation}
\label{map to Z2}
A \xrightarrow{\pi_0(F_\C)} \Pic_{br}(\E) =\Pic(\E) \times A  \xrightarrow{p}  \Pic(\E) \to \Pic(\E)/\Pic_{int}(\E) \cong \mathbb{Z}_2,
\end{equation}
where $p$ is the projection on $\Pic(\E)$. 
Its kernel consists of integral extensions, since $\Pic_{int}(\sVect)\cong \{ 1 \}$.  It remains to check that \eqref{phi}
is surjective. For this, it suffices to check that any homomorphism $\phi: A_0 \to \Pic(\sVect)\cong \mathbb{Z}_2$  gives rise to
a central $A_0$-extension of $\sVect$. Note that any Ising category is a central $\mathbb{Z}_2$-extension of $\sVect$
and so gives a monoidal $2$-functor $\mathbb{Z}_2 \to \uuPic(\sVect)$. Composing this with $\phi$, we get
a monoidal $2$-functor $A_0 \to \uuPic(\sVect)$ and, hence, a central $A_0$-extension $\sVect \hookrightarrow {\D}$.
As explained in Section~\ref{subs Kunneth}, this gives rise to a  minimal extension of $\E$
by taking the center of $\D$.
\end{proof}

\begin{theorem}
\label{factors}
Let $\E=\Rep(A,\,t)$ be a 
super-Tannakian category. The filtration \eqref{filtration} of $\Mext(\E)$ has factors
\begin{eqnarray*}
 \Mext_{triv}(\E) &\cong& \Coker \left( H^1(A,\, \widehat{A}) \xrightarrow{\kappa^t} H^3_{ab}(A,\,\kk^\times) \right),\\
 \Mext_{pt}(\E)/ \Mext_{triv}(\E) &\cong&  \Ker \left(  H^2_{ab}(A,\, \widehat{A})^\eps \xrightarrow{\theta^t}  \Hom(A_2/ \langle t \rangle,\,\kk^\times) \right), \\
 \Mext_{int}(\E)/ \Mext_{pt}(\E)&\cong&  \Hom({\wedge}^3 A ,\, \kk^\times), \\
  \Mext(\E) /  \Mext_{int}(\E) 
  &\cong&
   \begin{cases}
    \Hom(A, \, \mathbb{Z}_2) & \text{if $\E$ is split,} \\
    0 & \text{otherwise.}
   \end{cases}
\end{eqnarray*}
\end{theorem}
\begin{proof}
This follows from Propositions~\ref{Mext triv}, \ref{Mext pt}, \ref{Mext int}, and \ref{Mext gen}.
\end{proof}

\begin{remark}
For $t=1$, i.e., when $\E$ is Tannakian, the factors in Theorem~\ref{factors} were computed in \cite{MN} and \cite{DS}.
\end{remark}

\section{Examples}
\label{examples section}

Recall that when $t$ is a unique up to an automorphism central element  of order $2$  of a group $G$,
we use notation $\Rep(G^f)$ instead of  $\Rep(G,\,t)$.

\subsection{$\Mext(\Rep(\mathbb{Z}_{2^n}^f))$}
\label{cyclic 2-group}

\begin{proposition}
\label{two factors}
Let $A=\mathbb{Z}_m \times \mathbb{Z}_n$ and let $t\in A$ be  such that $\langle t \rangle$
is not a direct factor. Then any minimal non-degenerate extension of $\Rep(A,\,t)$ is pointed.
\end{proposition}
\begin{proof}
In this case, ${\wedge}^3 A =0$,  so the result follows from Proposition~\ref{Mext int}.
\end{proof}

For $n\geq 2$ let  $\E_n = \Rep(\mathbb{Z}_{2^n}^f)$.
By Proposition~\ref{two factors}, $\Mext(\E_n) = \Mext_{pt}(\E_n)$. It follows from Propositions~\ref{Mext triv} and \ref{Mext pt}
that there is an exact sequence
\begin{equation}
\label{sequence for cyclic}
0 \to  H^1(\mathbb{Z}_{2^n},\, \mathbb{Z}_{2^n}) \to H^3_{ab}(\mathbb{Z}_{2^n},\,\kk^\times) \to
\Mext(\E_n)  \xrightarrow{\lambda}  H^2_{ab}(\mathbb{Z}_{2^n},\, \mathbb{Z}_{2^n}) \to 0,
\end{equation}
where $\lambda$ assigns to the minimal extension $\Rep(\E_n) \hookrightarrow \C$
the cohomology class of the extension $0\to \mathbb{Z}_{2^n}  \to \Inv(\C) \to \mathbb{Z}_{2^n}  \to 0$.
Since $ H^3_{ab}(\mathbb{Z}_{2^n},\,\kk^\times) \cong \mathbb{Z}_{2^{n+1}}$ and 
$H^2_{ab}(\mathbb{Z}_{2^n},\, \mathbb{Z}_{2^n})= \mathbb{Z}_{2^n}$, the above sequence becomes
\begin{equation}
\label{sequence for cyclic precise}
0 \to  \mathbb{Z}_2 \xrightarrow{\alpha}
\Mext( \E_n) \xrightarrow{\lambda} \mathbb{Z}_{2^n} \to 0.
\end{equation}

Minimal non-degenerate extensions of $\E_n$ can be explicitly described as follows.
Recall \cite{EGNO, JS} that a pointed braided fusion category $\C$ with the group $\Inv(\C)=A$ of isomorphism 
classes of invertible objects  is determined up to an equivalence by the quadratic form $q:A \to \kk^\times$, where  $q(X) =c_{X,X}$.
In this case, we denote $\C=\C(A,\,q)$.

For each $m\geq 0$  and a primitive $2^{m+1}$th root of unity  $\xi$  we define a non-degenerate quadratic form
\begin{equation}
\label{qform qxi}
q_\xi: \mathbb{Z}_{2^m} \to \kk^\times,  \qquad q_\xi(j) =\xi^{{j^2}}\quad \text{for  all} \quad j \in \mathbb{Z}_{2^m}.
\end{equation} 
The non-degenerate pointed braided fusion category $\C(\mathbb{Z}_{2^m},\, q_\xi)$ is a minimal extension of~$\E_n$.

For each $k=0,1,\dots, n$ and a $2^{2n-k+1}$th root of unity $\zeta$  let 
\begin{equation}
\label{Mkzeta}
\M_{k,\zeta} = \C(\mathbb{Z}_{2^k} , q_{-\zeta^{-2^{2(n-k)}}}  )\bt \C(\mathbb{Z}_{2^{2n-k}},\, q_\zeta)
\end{equation}
Again, this is a pointed non-degenerate braided fusion category.

\begin{proposition}
\label{Mkzeta is minext}
For all $k$ and  $\zeta$ as above,
there is a non-degenerate minimal extension $\E_n \hookrightarrow \M_{k,\zeta}$ given by the  group homomorphism
\begin{equation}
\label{iotak}
\iota_k : \mathbb{Z}_{2^n} \to \mathbb{Z}_{2^{k}}  \times \mathbb{Z}_{2^{2n-k}}   : j \mapsto (j,\, j^{2^{n-k}} ).
\end{equation}
\end{proposition}
\begin{proof}
Clearly, \eqref{iotak} is an injective group homomorphism.  Let $q_{k, \zeta}: \mathbb{Z}_{2^k} \times \mathbb{Z}_{2^{2n-k}}  \to \kk^\times$
denote the quadratic form  corresponding to  $\M_{k,\zeta}$. We have
\[
q_{k, \zeta} (\iota_k(1)) =  -\zeta^{-2^{2(n-k)}} \cdot \zeta^{(2^{n-k})^2} =-1.
\]
Viewing \eqref{iotak}  as a homomorphism of metric groups,  where $\mathbb{Z}_{2^n}$ is equipped with a quadratic character $q(l)=(-1)^l$,
we obtain a braided tensor embedding $\E_n\hookrightarrow \M_{k,\zeta}$,
i.e., a minimal extension of $\E_n$.
\end{proof}

\begin{remark}
\label{properties of Mkzetas}
\begin{enumerate}
\item[(1)]  By definition, $\M_{0, \zeta} = \C(\mathbb{Z}_{2^{2n}},\,q_\zeta)$ is a cyclic minimal extension of~$\E_n$.
\item[(2)]  For all $k$ and $\zeta$, the largest order of the root of unity that occurs as a value of $q_{k, \zeta}$ is $2^{2n-k+1}$.
\item[(3)] For all $k=0,\dots, n-1,$ the square of $\M_{k,\zeta}$ in $\Mext(\E_n)$ is $\M_{k+1,\zeta^2}$ (this is a straightforward
computation using the definition of the product of minimal extensions). 
\end{enumerate}
\end{remark}

\begin{proposition}
\label{Z2n Mext}
$\Mext(\Rep(\mathbb{Z}_{2^n}^f) \cong \mathbb{Z}_{2^{n+1}}$  with any  $\C(\mathbb{Z}_{2^{2n}},\,q_\zeta)$ as a generator. 
\end{proposition}
\begin{proof}
We need to show that the exact sequence \eqref{sequence for cyclic precise} does not split. Observe that, 
for any primitive $2^{n+1}$th root of unity $\zeta$, the minimal extension
\[
\E_n \hookrightarrow  \M_{n,\zeta} = \C(\mathbb{Z}_{2^n},\,q_{-\zeta^{-1}}) \bt \C(\mathbb{Z}_{2^n},\,q_\zeta),
\]
where $\E_n$ is embedded diagonally, is the generator of $\Ker(\lambda)\cong \mathbb{Z}_2$ in \eqref{sequence for cyclic precise}
(in particular, its class in $\Mext(\E_n)$ does not depend on the choice of $\zeta$).  Thus,
it suffices to check that this minimal extension has a square root.  But this follows from Remark~\ref{properties of Mkzetas}(3).
\end{proof}

\begin{corollary} 
The kernel of the homomorphism 
$\Mext(\Rep(\mathbb{Z}_{2^n}^f)) 
\to \Mext(\sVect)$
is isomorphic to $\mathbb{Z}_{2^{n-2}}$.
\end{corollary}

\begin{remark}
The minimal extensions of $\Rep(\mathbb{Z}_{4}^f)$ and $\Rep(\mathbb{Z}_{8}^f)$ were listed in
in \cite[Tables XIV and XV]{LKW2}. Our  description of their groups of minimal extensions is consistent
with these tables and with the results of \cite{ABK} and \cite{W}.

For $n=2$ Proposition~\ref{Z2n Mext} says that $\Mext(\Rep(\mathbb{Z}_{4}^f)) \cong \mathbb{Z}_{8}$. This disagrees
with  \cite[Example 7.17]{VR}, where it is claimed that $|\Mext(\Rep(\mathbb{Z}_{4}^f))|=32$.  Our explanation 
of this discrepancy is that
\cite{VR} counts equivalence classes of  $\mathbb{Z}_2$-crossed braided extensions of  $\Z(\Rep(\sVect))$ whose
equivariantization is a minimal non-degenerate extension of $\Rep(\mathbb{Z}_{4}^f)$. However, all such 
$\mathbb{Z}_2$-crossed braided  extensions  lead to the same element of $\Mext(\Rep(\mathbb{Z}_{4}^f)$, namely, 
to the identity extension
$\Rep(\mathbb{Z}_{4}^f) \hookrightarrow \Z(\Rep(\mathbb{Z}_{4}))$.
\end{remark} 

\subsection{$\Mext({\Rep(\mathbb{Z}_{2} \times \mathbb{Z}_2^f)})$}
\label{Z2Z2 example}

Let $\E=\Rep(\mathbb{Z}_2\times \mathbb{Z}_2^f)$.  Using Theorem~\ref{factors}, we see that
the factors of the canonical filtration \eqref{filtration} of $\Mext(\E)$ are
\begin{eqnarray*}
\Mext_{triv}(\E)  &\cong& \mathbb{Z}_4 \times \mathbb{Z}_2, \\
\Mext_{pt}(\E)/ \Mext_{triv}(\E)  &=& \mathbb{Z}_2^2, \\
\Mext_{int}(\E) /  \Mext_{pt}(\E) &=&  0, \\
\Mext_{}(\E)/ \Mext_{int}(\E)  &=& \mathbb{Z}_2^2.
\end{eqnarray*}
Therefore, $|\Mext_{}(\E)|=128$. The canonical homomorphism 
\[
w: \Mext(\E) \to \Mext(\sVect) \cong \mathbb{Z}_{16},
\]
defined in \eqref{wGt}, is split surjective and, hence, 
\begin{equation}
\label{Z16 times Kerw}
\Mext(\E)  \cong \mathbb{Z}_{16} \times \Ker(w),
\end{equation}
where $|\Ker(w)|=8$.

\begin{example}
\label{ptd in Kerw}
The following categories are non-degenerate minimal extensions of $\E=\sVect\bt \sVect$ lying in $\Ker(w)$:
\begin{equation}
\label{M1M2M3}
\M_1(i) =\Z(\C(\mathbb{Z}_2,\, q_i) \bt \C(\mathbb{Z}_2,\, q_i)), \qquad \M_2(\xi) = \Z(\C(\mathbb{Z}_4,\, q_\xi)),\qquad \M_3(\I) =\Z(\I),
\end{equation}
where $i$ and $\xi$ are primitive $4$th and $8$th roots of unity, respectively, and $\I$ is an Ising braided fusion category.
We use the notation introduced in \eqref{qform qxi}.
In each of these three cases, there is a unique embedding of $\E$ (note that $\Z(\C)\cong \C \bt \C^\rev$ for any non-degenerate 
braided fusion category $\C$).

We have the following equalities in $\Mext(\E)$ :
\begin{equation}
\label{pm xi}
\M_1(i) = \M_1(-i)  \quad \text{and} \quad   \M_2(\xi) = \M_2(\xi') \quad \text{if and only if $\xi'=\pm \xi$}.
\end{equation}
\end{example}

\begin{lemma}
\label{Z4}
$\Ker(w) \cap \Mext_{pt}(\E)\cong \mathbb{Z}_{4}$ with a generator $\M_2(\xi)$.
\end{lemma}
\begin{proof}
Since $\Ker(w)$ contains a non-integral extension $\M_3(\I)$, we see that $|\Ker(w) \cap \Mext_{pt}(\E)|=4$.

The identity minimal extension $\Z(\E)$ (respectively, $\M_1(i)$ and  $\M_2(\xi)$)  is a pointed minimal extension of $\E$
characterized by the property that the largest order of the root of unity that occurs as a value of the corresponding quadratic form
is $2$ (respectively, $4$ and $8$). Since for the square of $\M_2(\xi)$ in $\Mext(\E)$ the quadratic form has values that are $4$th roots of unity,
it follows that $\M_1(i)$  is a square in $\Ker(w)$ and  the statement follows.
\end{proof}

\begin{lemma}
\label{Z8}
$\Ker(w)  \cong \mathbb{Z}_{8}$ with a generator $\M_3(\I)$ for any Ising category $\I$.
\end{lemma}
\begin{proof}
For a braided Ising fusion category $\I$ the values of the canonical twist on its simple objects are $1,\, -1$, and a primitive $16$th root of unity $\zeta$,
see \cite[Appendix B]{DGNO2}.   Therefore, the values of the twist on the simple objects of the integral  part $(\I \bt \I)_{pt}$ of $\I \bt \I$ are $1,\,1,\,-1,\,-1$, 
and $\zeta^2$. By definition of the tensor product of minimal extensions, the tensor square of $\Z(\I)$ in $\Mext(\E)$ is the de-equivariantization of 
$(\I \bt \I)_{pt} \bt (\I^\rev \bt \I^\rev)_{pt}$ (viewed as a subcategory of $\Z(\I)^{\bt 2}$ ) by the diagonal Tannakian subcategory of $\E \bt \E$. 
The result contains simple  objects with the twist (quadratic form) values being primitive $8$th roots of unity. This means that 
\[
\Z(\I)^{\boxdot 2} = \Z(\C(\mathbb{Z}_4,\, q_{\zeta^2})) = \M(\zeta^2),
\]
so that $\M_3(\I)$ has order $8$ in $\Mext(\E)$ by Lemma~\ref{Z4}.
\end{proof}

\begin{corollary}
\label{128 group}
$\Mext(\Rep(\mathbb{Z}_2\times \mathbb{Z}_2^f)) \cong  \mathbb{Z}_{16} \times \mathbb{Z}_8$. The extensions  $\Z(\Rep(\mathbb{Z}_2 )) \bt \I_1$
and $\Z(\I_2)$, where $\I_1,\, \I_2$ are any Ising braided fusion categories, can be taken 
as generators of the cyclic factors.

The central charge homomorphism \eqref{wGt} is identified with the projection on the first factor.
\end{corollary}
\begin{proof}
This follows from \eqref{Z16 times Kerw} and  Lemma~\ref{Z8}.
\end{proof}

\begin{corollary}
\label{128 group pt}
$\Mext_{pt}(\Rep(\mathbb{Z}_2\times \mathbb{Z}_2^f))\cong  \mathbb{Z}_{8} \times \mathbb{Z}_4$. 
\end{corollary}
\begin{proof}
We have  seen that $|\Mext_{pt}(\Rep(\mathbb{Z}_2\times \mathbb{Z}_2^f))|=32$.  Since the minimal extensions that project
on Ising categories in $\Mext(\sVect)$ are non-integral, we conclude that $\Mext_{pt}(\Rep(\mathbb{Z}_2\times \mathbb{Z}_2^f))$
is a subgroup of $\mathbb{Z}_{8} \times \mathbb{Z}_8$, which implies the result. 
\end{proof}

\begin{remark}
All $128$ minimal extensions of $\Rep(\mathbb{Z}_2\times \mathbb{Z}_2^f)$
were listed in \cite[Tables XVI-XIX]{LKW2}. Our contribution is the computation of the group structure 
of $\Mext(\Rep(\mathbb{Z}_2\times \mathbb{Z}^f))$.
The filtration of this group  and its integral part  found in this Section are consistent with these tables.
A description of this group was also given recently in \cite[Section V.B and Appendix I]{ABK} and in 
\cite[Table V]{BCHM}, where the eight elements of the group $\Ker(w)$ were explicitly listed.
\end{remark} 

\bibliographystyle{ams-alpha}

\begin{thebibliography}{A} 

\bibitem{ABK} D.~Aasen, P.~Bonderson,  C.~Knapp,
\textit{Characterization and classification of fermionic symmetry enriched topological phases},
eprint arXiv:2109.10911v2  (2021).

\bibitem{BCHM} M.~Barkeshli, Y.-A.~Chen, P.-S.~Hsin, N.~Manjunath,
\textit{Classification of (2+1)D invertible fermionic topological phases with symmetry},
eprint arXiv:2109.11039 (2021).


%





\bibitem{BGHNPRW}
P.~Bruillard,  C.~Galindo,T.~Hagge, S.-H.~Ng, J.~Plavnik,  E.~Rowell, Z.~Wang, 
\textit{Fermionic modular categories and the 16-fold way},
J.\ Math.\ Phys.\ \textbf{58}, 041704 (2017).


\bibitem{C}  G.~Carnovale,
{\em The Brauer group of modified supergroup algebras},
J.\ Algebra \textbf{305} (2006), no. 2, 993--1036.   



%
%


\bibitem{DN1}  A.~Davydov, D.~Nikshych, 
{\em The Picard crossed module of a braided tensor category},
Algebra and Number Theory, \textbf{7} (2013), no. 6, 1365--1403.

\bibitem{DN2}  A.~Davydov, D.~Nikshych, 
{\em Braided Picard groups and  graded extensions of braided tensor categories},
Sel.\ Math.\ New Ser.\textbf{27}, 65 (2021). 


\bibitem{DNO}  A.~Davydov, D.~Nikshych, V.~Ostrik,
{\em On the structure of the Witt group of non-degenerate braided fusion categories},
Selecta Mathematica \textbf{19} (2013), no. 1, 237--269.

\bibitem{DS} A.~Davydov, D.~Simmons,
{\em On Lagrangian algebras in group-theoretical braided fusion categories},
J.\ Algebra, \textbf{471} (2017),  149--175.



\bibitem{DGPRZ} C.~Delaney, C.~Galindo, J.~Plavnik, E.~Rowell, Q.~Zhang,
{\em Braided zesting and its applications},
Comm.\ Math.\ Physics, \textbf{386} (2021), 1-55. 


\bibitem{D} P.~Deligne, 
{\em Cat\'egories tensorielles},  Mosc. Math. J.  {\bf 2} (2002),  no. 2, 227--248.

\bibitem{DNR} C.~Dong, S.-H.~Ng, L.~Ren,
{\em Orbifolds and minimal modular extensions},
eprint arXiv:2108.05225 (2021).


%

\bibitem{DGNO1} V.~Drinfeld, S.~Gelaki, D.~Nikshych,  V.~Ostrik,
\textit{Group-theoretical properties of nilpotent 
modular categories},  e-print arXiv:0704.0195v2 [math.QA] (2007).

\bibitem{DGNO2}  V.~Drinfeld, S.~Gelaki, D.~Nikshych,  V.~Ostrik.
\textit{On braided fusion categories I}, 
Selecta Mathematica,  \textbf{16}  (2010), no.\ 1,  1~{-}~119.

\bibitem{EM1} S.~Eilenberg, S.~MacLane, 
\textit{Cohomology Theory of Abelian Groups and Homotopy Theory II}, 
Proc.\ Natl.\ Acad.\ Sci.\  USA  \textbf{36}(11) (1950), 657-663.

\bibitem{EM2} S.~Eilenberg, S.~MacLane, 
\textit{ On the groups $H(\Pi,n)$, II,} 
Annals of Mathematics, Second Series, Vol. 60, No. 1 (Jul., 1954), pp. 49-139



\bibitem{EGNO} P.~Etingof, S.~Gelaki, D.~Nikshych, V.~Ostrik, 
{\em  Tensor categories}, 
Mathematical Surveys and Monographs,
\textbf{205}, American Mathematical Society (2015).



\bibitem{ENO}   P.~Etingof, D.~Nikshych, V.~Ostrik,
{\em  Fusion categories and homotopy theory}, 
Quantum Topology,    \textbf{1} (2010), no.\ 3,  209-273.

%

\bibitem{GVR} C.~Galindo, C.F.~Venegas-Ramirez,
\textit{Categorical fermionic actions and minimal modular extensions},
arXiv preprint arXiv:1712.07097  (2017).




\bibitem{GNN} S.~Gelaki, D.~Naidu,  D.~Nikshych,
{\em  Centers of graded fusion categories}, 
Algebra and Number Theory, \textbf{3} (2009), no.\ 8, 959-990. 

\bibitem{GN} S.~Gelaki, D.~Nikshych,
Nilpotent fusion categories. 
\textit{Advances in Mathematics} \textbf{217} (2008), no. 3, 1053--1071.

\bibitem{GS} S.~Gelaki,   D.~Sebbag,  
On finite non-degenerate braided tensor categories with a Lagrangian subcategory. 
preprint arXiv:1703.05787 [math.QA] (2017).


\bibitem{JMPP}  C.~Jones, S.~Morrison, D.~Penneys, J.~Plavnik,
\textit{Extension theory for braided-enriched fusion categories},
preprint, arXiv:1910.03178 [math.QA] (2019).

\bibitem{JS} A.~Joyal, R.~Street,
\textit{Braided tensor categories},
Adv.\ Math., \textbf{102}, 20-78 (1993).


%

\bibitem{K} G. Karpilovsky,
\textit{The Schur Multiplier}, 
Oxford University Press (1987).

\bibitem{Kit} A.~Kitaev,
\textit{Anyons in an exactly solved model and beyond}, 
Ann.\ Physics,  \textbf{32}  (2006),  no. 1, 2--111.

 \bibitem{LKW1} T.~Lan, L.~Kong, X.-G.~Wen,
\textit{Modular extensions of unitary braided fusion categories
and 2+1D topological/SPT orders with symmetries},
Comm.\ Math.\ Phys. \textbf{351}, 709-739, (2017).

 \bibitem{LKW2} T.~Lan, L.~Kong, X.-G.~Wen,
\textit{Classification of (2+1)-dimensional topological order and symmetry-protected topological order for bosonic and fermionic systems with on-site symmetries}
Phys\. Rev.\  B \textbf{95}, 235140 (2017).


\bibitem{MN} G.~Mason, S.-H.~Ng,
\textit{Group cohomology and gauge equivalence of some twisted quantum doubles}, 
Trans.\ AMS \textbf{353}, 3465-3509 (2001).


\bibitem{Mu} M.~M\"uger, \textit{On the structure of modular categories},
Proc.\ Lond.\ Math.\ Soc., \textbf{87} (2003), 291-308.

%
%


 
 %



\bibitem{VR} C.F.~Venegas-Ramirez,
\textit{Minimal modular extensions for super-Tannakian categories},
arXiv preprint arXiv:1908.07487 (2019).


\bibitem{W} C.~Wang,
\textit{Braiding statistics and classification of two-dimensional charge-$2m$ superconductors},
Phys.\ Rev.\ B \textbf{94}, 085130 (2016).

\end{thebibliography}

\end{document}